%% file: Step_size_control.tex
\documentclass[11pt]{amsart}

\usepackage{multirow}
\usepackage{amssymb}
\usepackage{rotating}
\usepackage{diagbox}
\usepackage{array}
\usepackage{graphicx}
\usepackage{subfig}
\usepackage{chngcntr}
\usepackage{apptools}
\usepackage{hhline}
\usepackage[bw,framed]{includes/mcode}
\usepackage{graphicx,color}
\graphicspath{{figs/}}
\newcolumntype{C}[1]{>{\centering\let\newline\\\arraybackslash\hspace{0pt}}m{#1}}
\include{includes/sb_macros}
\newtheorem{assumption}[definition]{Assumption}

\usepackage{boxedminipage}
\allowdisplaybreaks

\def\Rinf{\R\cup \{+\infty\}}

\def\tl{\widetilde{\l}}
\usepackage{stmaryrd}

\newcommand{\til}[1]{\widetilde{#1}}
\newcommand{\h}[1]{\widehat{#1}}

\begin{document}
\title[ADMM with variable step sizes]{Alternating direction method of multipliers with variable step sizes}
\author{S\"{o}ren Bartels}
\address{Department of Applied Mathematics, Mathematical Institute, University of Freiburg, Hermann-Herder-Str. 9, 79104 Freiburg i. Br., Germany}
\email{bartels@mathematik.uni-freiburg.de}
\author{Marijo Milicevic}
\address{Department of Applied Mathematics, Mathematical Institute, University of Freiburg, Hermann-Herder-Str. 9, 79104 Freiburg i. Br., Germany}
\email{marijo.milicevic@mathematik.uni-freiburg.de}
\date{\today}

\maketitle

\begin{abstract}
The alternating direction method of multipliers (ADMM) is a flexible method to solve a large class
of convex minimization problems. Particular features are its unconditional convergence with respect
to the involved step size and its direct applicability. This article deals with the ADMM
with variable step sizes and devises an adjustment rule for the step size relying on the monotonicity
of the residual and discusses proper stopping criteria. The numerical experiments show significant
improvements over established variants of the ADMM.
\end{abstract}

\section{Introduction}

The development of iterative schemes for convex minimization problems is a fundamental and
challenging task in applied mathematics with a long history reflected in a number of articles, e.g., in
~\cite{bt:09, bv:04, chambolle:04, cp:11, fg:83, glowinski:84, gl:89, go:09, hik:03, lm:79, nesterov:83, nesterov:05, rockafellar:76}.
These include gradient descent methods, semi-smooth Newton methods, (accelerated) primal-dual-methods,
dual methods, Br\`{e}gman iteration and operator splitting methods.
Here, we aim at devoloping a strategy for an automated adjustment of the step size of the
alternating direction method of multipliers (ADMM), which is an operator splitting method,
motivated by the fact that its performance is known to critically depend on the involved step size.

We consider convex variational problems of the form
\[
\inf_{u \in X} F(Bu) + G(u)
\]
that arise in various applications from partial differential equations, mechanics, imaging and economics, e.g.,
the $p$-Laplace equation, the ROF model for image denoising, obstacle problems and convex programming. We assume that
possible natural constraints are encoded in the objective functionals via indicator functionals. One way to solve this
minimization problem is to introduce an auxiliary variable $p=Bu$ which leads to the constrained minimization problem
\[
\inf_{(u,p) \in X \times Y} F(p) + G(u) \quad \text{subject to} \quad p = Bu.
\]
With a Lagrange multiplier ~$\l$ and a positive step size $\tau >0$
the corresponding augmented Lagrangian reads
\[
\cL_\tau(u,p;\l) = F(p) + G(u) + (\l,Bu-p)_Y + \frac{\tau}{2} \|Bu-p\|_Y^2,
\]
which has been introduced by Hestenes in ~\cite{hestenes:69} and Powell in ~\cite{powell:69}. The augmented
Lagrangian method (ALM) minimizes ~$\cL_\tau$ with respect to ~$(u,p)$ jointly and then updates the Lagrange
multiplier ~$\l$. Since the joint minimization with respect to ~$u$ and ~$p$ is almost as hard as the original
problem, the idea of the alternating direction method of multipliers (ADMM) is to decouple the minimization
and to minimize ~$\cL_\tau$ with respect to ~$u$ and ~$p$ successively and then update the
Lagrange multiplier. In this way one can benefit from the particular features of the objective functionals
~$F$ and ~$G$ in the sense that the separate minimization problems can often be solved directly.

The ADMM was first introduced by Glowinski and Marroco in ~\cite{gm:75} and
Gabay and Mercier in ~\cite{gm:76}. For a comprehensive discussion on ALM and ADMM and applications in partial
differential equations, consider, for instance, ~\cite{glowinski:84}, and for the connection of ADMM to other
splitting methods, particularly the Douglas-Rachford splitting method (DRSM) ~\cite{dr:56}, see, e.g., ~\cite{gabay:83,eckstein:89}.
In the recent literature versions of ADMM for convex minimization problems with more than two primal and/or auxiliary
linearly constrained variables have also been analyzed, see, e.g., ~\cite{dhyz:16,hyz:14}.

He and Yuan establish in ~\cite{hy:12} an ~$\mathcal{O}(1/J)$ convergence rate in an ergodic sense of a quantity
related to an optimality condition which is based on a variational inequality reformulation of the constrained minimization
problem. In ~\cite{hy:15}, they prove an ~$\mathcal{O}(1/J)$ convergence rate for the residual of the ADMM. Shen and Xu
prove in ~\cite{sx:14} an $\mathcal{O}(1/J)$ convergence rate in an ergodic sense for a modified ADMM proposed by
Ye and Yuan in ~\cite{yy:07} which augments the original ADMM with fixed step size by an additional extrapolation
of the primal and dual variable. In ~\cite{gosb:14},
motivated by the work of Nesterov ~\cite{nesterov:83}, Goldstein et al. consider an accelerated version of ADMM (Fast-ADMM)
and prove an ~$\mathcal{O}(1/J^2)$ convergence rate for the objective value of the dual problem of the constrained
minimization problem under the assumption that both objective funtionals are strongly convex and ~$G$ is quadratic.
Furthermore, they prove an ~$\mathcal{O}(1/J^2)$ convergence rate for the residuals if in addition ~$B$ has full row rank
and also propose a Fast-ADMM with restart for the case of ~$F$ and ~$G$ being only convex for which, however,
a convergence rate ~$\mathcal{O}(1/J^2)$ could not been proven.
Recently, Deng and Yin established in ~\cite{dy:16} the linear convergence of a generalized ADMM
under a variety of different assumptions on the objective functionals and the operator ~$B$.
Particularly, they derive an explicit upper bound for the linear convergence rate of the ADMM and, by optimizing the rate with respect
to the step size, obtain an optimized step size. However, experiments reveal that the optimized step size often leads
to a pessimistic convergence rate. Furthermore, the computation of the optimized step size requires the knowledge of
the strong convexity constant of ~$G$ and the Lipschitz constant of ~$\nabla G$ as well as the computation of the
minimal eigenvalue of the operator ~$B'B$. In many of the cited results, the convergence rate of the ADMM critically depends on the
dimension of ~$X$ via the operator norm of ~$B$ and an inf-sup-condition associated to ~$B$.

Since the convergence rate of the ADMM depends on the step size ~$\tau$ it seems reasonable to
consider a variation of the step size. He and Yang have already studied
the ADMM with variable step sizes in ~\cite{hy:98} under the assumption that ~$F$ and ~$G$ are continuously differentiable.
They prove termination of the algorithm for monotonically decreasing or monotonically increasing step sizes and prove
convergence for fixed step sizes. In ~\cite{km:98}, Kontogiorgis and Meyer also consider the ADMM with variable step sizes
and prove convergence provided that the step sizes decrease except for finitely many times. However, the suggested
strategy in varying the step sizes is tailored to the problems discussed in their work. In ~\cite{hyw:00,hlw:03} He et al.
consider the ADMM with self-adaptive step sizes and prove convergence provided that the step size is uniformly bounded
from above and below and that the sequence of difference quotients
of the step size is summable. Their proposed adjustment rule for the step size aims at balancing the two components
of the residual of the ADMM. In ~\cite{hlhy:02}, He et al. consider an inexact ADMM where certain proximal terms are
added to the augmented Lagrangian functional ~$\cL_\tau$ and the
subproblems are only solved approximately. They also allow for variable step sizes and proximal parameters and prove
convergence under a summability condition of the difference quotients of the step sizes and proximal parameters
and suggest to adjust the step size in such a way that the two contributions of the residual of the ADMM balance out.

In this paper we aim at devising a general automatic step size adjustment strategy. The adjustment of the step size is based on the
observation that the residual of the ADMM decreases as long as the sequence of step sizes is non-increasing. Particularly, we prove
that the residual reduces by a factor of at least $\gamma<1$ as long as the step size decreases and under the assumption
that ~$\nabla G$ is Lipschitz continuous and ~$B'$ is injective with bounded left-inverse. More precisely, we propose
the following strategy:

\begin{enumerate}
\item Choose a feasible initialization $(u^0,\l^0)$, a large step size ~$\tau_0$ and a contraction factor $\g \in (0,1)$.
\item Minimize ~$\cL_\tau$ with respect to ~$p$, then minimize ~$\cL_\tau$ with respect to ~$u$ and finally update ~$\l$.
\item Check if the residual is decreased by the factor ~$\g$. If this is not the case decrease the step size.
\item If the current step size is smaller than a chosen lower bound, restart the algorithm with a larger ~$\g \in (0,1)$.
Otherwise continue with ~(2).
\end{enumerate}

We furthermore address the choice of accurate stopping criteria and propose a stopping criterion that controls
the distance between the primal iterates and the exact solution if strong coercivity is given.

The paper is organized as follows. In Section ~\ref{sec:preliminaries} we state some basic notation that we use and briefly
introduce the finite element spaces we use when applying the general framework to convex model problems. Section ~\ref{sec:admm} is devoted to the analysis of ADMM with variable step sizes. In Subsection ~\ref{subsec:min problem} we present the minimization problem
and the corresponding saddle-point formulation on which the ADMM is based. The ADMM with variable step sizes is then recalled in Subsection ~\ref{subsec:convergence of admm} and a convergence proof, which is similar to that of Kontogiorgis and Meyer in ~\cite{km:98}, is given. Additionally, we show that the residual controls the distance between iterates and a saddle-point if strong coercivity is given. The monotonicity of the residual of the ADMM with variable step sizes, which is proven in Subsection ~\ref{subsec:monotonicity and convergence rate}, is a crucial property of the method to ensure the termination of the scheme even if the step size is being decreased. Furthermore, it directly implies a sublinear convergence rate of the method without any additional conditions on the functionals. In Subsection ~\ref{subsec:linear convergence} we deal with the linear convergence of the ADMM with variable step sizes which motivates to adjust the step size according to the contraction properties of the residual. Subsequently, we make our approach more precise and present the Variable-ADMM in Subsection ~\ref{subsec:adaptivity}. In Section ~\ref{sec:numerical experiments} we apply our algorithm to the obstacle problem and the $TV$-$L^2$ minimization (or ROF) problem and compare its performance to the classical ADMM with fixed step size and the Fast-ADMM proposed by Goldstein et al. in ~\cite{gosb:14}, which we specify in the appendix. Here, we also discuss stopping criteria used in the literature that may fail to lead to an accurate approximation. Finally, we give a conclusion in Section ~\ref{sec:conclusion}.

\section{Preliminaries}\label{sec:preliminaries}

\subsection{Notation}

We consider two Hilbert spaces $X,Y$ equipped with inner products $(\cdot,\cdot)_X$ and
$(\cdot,\cdot)_Y$, respectively, and identify their duals, denoted by ~$X'$ and ~$Y'$, with ~$X$ and ~$Y$, respectively.
For a linear operator $B:X \to Y$ we denote by $B':Y'\to X'$ its adjoint.
We furthermore let $\O\subset\R^d$, ~$d=2,3$, be a bounded polygonal
Lipschitz domain. The $L^2$-norm on ~$\O$ is denoted by ~$\|\cdot\|$ and is induced by the scalar product
\[
(v,w) := \int_\O v \cdot w \dv{x}
\]
for scalar functions or vector fields $v,w \in L^2(\O;\R^r)$, $r \in \{1,d\}$, and we write ~$|\cdot|$
for the Euclidean norm. We use for arbitrary sequences $(a^j)_{j \in \N}$ and step sizes $\tau_j>0$
the backward difference quotient
\[
d_t a^{j+1} := \frac{a^{j+1}-a^j}{\tau_{j+1}}.
\]
Using this definition we will also work with
\[
d_t^2 a^{j+1} = \frac{d_t a^{j+1} - d_t a^j}{\tau_{j+1}} = \frac{\frac{a^{j+1}-a^j}{\tau_{j+1}}-\frac{a^j-a^{j-1}}{\tau_j}}{\tau_{j+1}}.
\]
Note that we have the discrete product rules
\begin{subequations}
\begin{align}
2 d_t a^{j+1} \cdot a^{j+1} &= d_t |a^{j+1}|^2 + \tau_{j+1} |d_t a^{j+1}|^2, \label{eq:product rule 2}\\
d_t (\tau_{j+1}^2 |a^{j+1}|^2) &= (d_t \tau_{j+1}^2) |a^j|^2 + \tau_{j+1}^2 d_t |a^{j+1}|^2. \label{eq:product rule 3}
\end{align}
\end{subequations}
Finally, throughout the paper ~$c$ will denote a generic, positive and mesh-independent constant.

\subsection{Finite element spaces}

We let ~$(\cT_h)_{h>0}$ be a family of
regular triangulations of ~$\O$ with mesh sizes $h=\max_{T \in \cT_h} h_T$ with ~$h_T$ being the diameter
of the simplex ~$T$. We further denote $h_{\min}=\min_{T \in \cT_h} h_T$. For a given triangulation ~$\cT_h$ the set
~$\cN_h$ contains the corresponding nodes and we consider the finite element spaces of continuous,
piecewise affine functions 
\[
\cS^1(\cT_h) := \bigl\{v_h\in C(\overline{\O}): \; v_h|_T \text{ affine for all }T\in \cT_h\bigr\}
\]
and of elementwise constant functions ~($r=1$) or vector fields ~($r=d$)
\[
\cL^0(\cT_h)^r := \big\{q_h\in L^\infty(\O;\R^r): \; q_h|_T \text{ constant for all } T\in \cT_h\big\}.
\]

Correspondigly, we denote by ~$\cT_\ell$, ~$\ell \in \N$, a triangulation of ~$\O$ generated from an initial triangulation
~$\cT_0$ by ~$\ell$ uniform refinements. The refinement level ~$\ell$ will be related to the mesh size ~$h$ by
~$h \sim 2^{-\ell}$. The set of nodes ~$\cN_\ell$ is then defined as before.
In our experiments we will use the discrete norm ~$\|\cdot\|_h$ induced by the 
discrete scalar product
\[
(v,w)_h := \sum_{z \in \cN_h} \b_z v(z) w(z)
\]
for $v,w \in \cS^1(\cT_h)$, where ~$\b_z=\int_\O \varphi_z \dv{x}$ and $\varphi_z \in \cS^1(\cT_h)$
is the nodal basis function associated with the node $z \in \cN_h$. This mass lumping will allow for the
nodewise solution of certain nonlinearities. We have the relation
\[
\|v_h\| \leq \|v_h\|_h \leq (d+2)^{1/2} \|v_h\|
\]
for all $v_h \in \cS^1(\cT_h)$, cf. \cite[Lemma 3.9]{sbartels:15}.
On $\cL^0(\cT_h)^r$ we will also consider the weighted $L^2$-inner product
\[
(\cdot,\cdot)_w := h^d (\cdot,\cdot)
\]
which has the property $\|q_h\|_w \leq c \|q_h\|_{L^1(\Omega)}$ due to an inverse estimate,
cf. ~\cite{bs:08}.

\section{Alternating direction method of multipliers}\label{sec:admm}

\subsection{Minimization problem and saddle-point formulation}\label{subsec:min problem}

We are given convex, proper, and lower-semicontinuous functionals $F:Y\to \Rinf$, 
$G:X\to \Rinf$, and a bounded linear operator $B:X\to Y$ such that the functional $I(\cdot)=F(B\cdot) + G(\cdot)$
is proper and coercive. We consider the minimization problem
\[
\inf_{u\in X} I(u) = \inf_{u\in X} F(Bu) + G(u).
\]
Upon introducing $p = Bu$ and choosing $\tau>0$
we obtain the equivalent, consistently 
stabilized saddle-point problem defined by
\[
\inf_{(u,p)\in X\times Y}\sup_{\l\in Y}\cL_\tau(u,p;\l)= F(p) + G(u) + (\l,Bu-p)_Y
+ \frac{\tau}{2} \|Bu-p\|_Y^2.
\]

\begin{remark}
If there exists a saddle-point $(u,p;\l)$ for ~$\cL_\tau$,
then ~$u$ is a minimizer of ~$I$ and $p=Bu$, cf. ~\cite[Chapter VI, Thm. 2.1]{glowinski:84}.
On the other hand, if ~$X$ and ~$Y$ are finite-dimensional and if ~$u \in X$ is a minimizer of
~$I$, the existence of a saddle-point $(u,p;\l)$ for ~$\cL_\tau$ can be proven 
by taking $p=Bu$, using ~\cite[Thm. 23.8]{rockafellar:70} and ~\cite[Thm. 23.9]{rockafellar:70},
and incorporating the fact that $(v,q) \mapsto \cL_\tau(v,q;\l)$ for fixed ~$\l$
is convex, proper, coercive and lower-semicontinuous and therefore admits a minimizer.
The characterizing optimality conditions for such a minimizer are satisfied by the pair $(u,p)$ if
~$\l$ is chosen properly and one deduces that $(u,p;\l)$ is a saddle-point for ~$\cL_\tau$ (see also
~\cite{glowinski:84,rockafellar:70}).
\end{remark}

In this paper we make the following assumption.

\begin{assumption}\label{ass:existence of sp}
There exists a saddle-point $(u,p;\l)$ for ~$\cL_\tau$.
\end{assumption}

Possible strong convexity of $F$ or $G$ is characterized by
nonnegative functionals $\varrho_F:Y\times Y \to \R$ and
$\varrho_G:X\times X \to \R$ in the following lemma.

\begin{lemma}[Optimality conditions]\label{lemma:opt cond}
A triple $(u,p,\l)$ is a saddle point for $\cL_\tau$ if and only if
$Bu=p$ and 
\[\begin{split}
\big(\l,q-p\big)_Y + F(p) + \varrho_F(q,p) 
& \le F(q),\\
-\big(\l,B(v-u)\big)_Y + G(u) + \varrho_G(v,u) 
& \le G(v),
\end{split}\]
for all $(v,q)\in X\times Y$. 
\end{lemma}

\begin{proof}
The variational inequalities characterize stationarity with 
respect to ~$u$ and ~$p$, respectively, i.e., that, e.g.,
$0\in \p_u \cL_\tau (u,p;\l)$. 
\end{proof}

For ease of presentation we introduce the symmetrized coercivity
functionals
\[
\h{\varrho}_G(u,u')= \varrho_G(u,u') + \varrho_G(u',u),
\quad \h{\varrho}_F(p,p')= \varrho_F(p,p') + \varrho_F(p',p).
\]

\subsection{Algorithm and convergence}\label{subsec:convergence of admm}

We approximate a saddle-point using the following iterative scheme which has been introduced
in ~\cite{gm:75,gm:76,glowinski:84} with fixed step sizes.

\begin{algorithm}[Generalized ADMM]\label{alg:alg_2}
Choose $(u^0,\l^0)\in X\times Y$ such that $G(u^0)<\infty$.
Choose $\overline{\tau} \geq \underline{\tau}>0$ and $\overline{R}\gg 0$ and set $j=1$. \\
(1) Set $\tau_1=\overline{\tau}$ and $R_0=\overline{R}$.\\
(2) Compute a minimizer $p^j\in Y$ of the mapping
\[
p\mapsto \cL_{\tau_j}(u^{j-1},p;\l^{j-1}). 
\]
(3) Compute a minimizer $u^j\in X$ of the mapping 
\[
u \mapsto \cL_{\tau_j} (u,p^j;\l^{j-1}).  
\]
(4) Update $\l^j = \l^{j-1} + \tau_j (Bu^j-p^j)$. \\
(5) Define 
\[
R_j = \bigl(\|\l^j-\l^{j-1}\|_Y^2 + \tau_j^2 \|B(u^j-u^{j-1})\|_Y^2 \bigr)^{1/2} .
\]
(6) Stop if ~$R_j$ is sufficiently small.\\
(7) Choose step size $\tau_{j+1}\in [\underline{\tau},\overline{\tau}]$.\\
(8) Set $j\to j+1$ and continue with~(2).
\end{algorithm}

\begin{remarks}
(1) A variant of the ADMM with variable step sizes has been proposed and analyzed
in ~\cite{km:98}. Therein, a more general scheme
with symmetric positive definite matrices ~$H_j$ is presented.
We will give a more compact proof of boundedness of the iterates and
termination of the algorithm related to ~\cite[Lem. 2.5, Lem. 2.6]{km:98} with $H_j=\tau_j I$.\\
(2) We call Algorithm ~\ref{alg:alg_2}
with fixed step sizes, i.e., $\overline{\tau}=\underline{\tau}$, simply ``ADMM".\\
(3) In Subsection ~\ref{subsec:adaptivity} we present a strategy for the adjustment of ~$\tau_j$
based on contraction properties and introduce the ``Variable-ADMM".
\end{remarks}

The iterates of Algorithm ~\ref{alg:alg_2} satisfy the following optimality
conditions.

\begin{lemma}[Decoupled optimality]\label{lemma:dec opt cond}
With $\tl^j := \l^{j-1} + \tau_j(Bu^{j-1}-p^j)$ the iterates $(u^j,p^j,\l^j)_{j=0,1,\dots}$ satisfy
for $j \geq 1$ the variational inequalities
\[\begin{split}
\big( \tl^j,q-p^j\big)_Y + F(p^j)
+ \varrho_F(q,p^j) &\le F(q), \\
-\big(\l^j,B(v-u^j)\big)_Y + G(u^j) 
+ \varrho_G(v,u^j) &\le G(v),
\end{split}\]
for all $(v,q)\in X\times Y$. In particular, $(u^j,p^j;\l^j)$ is a saddle-point of ~$\cL_\tau$ if and only if
$\l^j-\l^{j-1}=0$ and $B(u^j-u^{j-1})=0$.
\end{lemma}

\begin{proof}
By step ~(1) in Algorithm ~\ref{alg:alg_2} we have $0 \in \partial_p\cL_{\tau_j}(u^{j-1},p^j;\l^{j-1})$
which is equivalent to the first variational inequality using the definition of ~$\tl^j$.
Step ~(2) implies $0 \in \partial_u\cL_{\tau_j}(u^j,p^j;\l^{j-1})$ which
is equivalent to the second variational inequality using the definition of ~$\l^j$ in step ~(3).
\end{proof}

We set $\tau_0:=\tau_1$ for ~$d_t \tau_1=0$ to be well-defined in the convergence proof.


\begin{theorem}[Convergence]\label{thm:conv}
Let $(u,p;\l)$ be a saddle point for $\cL_\tau$.
Suppose that the step sizes satisfy the
monotonicity property 
\[
0 < \underline{\tau} \leq \tau_{j+1} \leq \tau_j
\]
for $j \geq 1$.
For the iterates $(u^j,p^j;\l^j)$, $j\geq 0$, of Algorithm~\ref{alg:alg_2},
the corresponding differences $\d_\l^j = \l-\l^j$, $\d_p^j = p-p^j$ and $\d_u^j = u-u^j$,
and the distance
\[
D_j^2=\|\d_\l^j\|_Y^2 + \tau_j^2 \|B \d_u^j\|_Y^2,
\]
we have that for every $J\ge 1$ it holds
\[\begin{split}
\frac{1}{2} D_J^2 + \sum_{j=1}^J \Bigl( \tau_j \bigl(\h{\varrho}_G(u,u^j) + \h{\varrho}_F(p,p^j) + \h{\varrho}_G(u^{j-1},u^j)\bigr) + \frac{1}{2}R_j^2\Bigr)
\le  \frac{1}{2} D_0^2.
\end{split}\]
In particular, $R_j \to 0$ as $j\to \infty$ and Algorithm~\ref{alg:alg_2} terminates.
\end{theorem}

\begin{proof}
Choosing $(v,q)=(u^j,p^j)$ in Lemma ~\ref{lemma:opt cond} and $(v,q)=(u,p)$ in Lemma
~\ref{lemma:dec opt cond} and adding corresponding inequalities we obtain
\begin{align*}
(\tl^j-\l,p-p^j)_Y + \h{\varrho}_F(p,p^j) &\leq 0, \\
(\l -\l^j,B(u-u^j))_Y + \h{\varrho}_G(u,u^j) &\leq 0.
\end{align*}
Adding the inequalities, inserting ~$\l^j$, and using that $Bu=p$ we obtain
\[
\h{\varrho}_F(p,p^j) + \h{\varrho}_G(u,u^j) \leq (\l-\l^j,B u^j-p^j)_Y + (\l^j-\tl^j,p-p^j)_Y.
\]
With $\l^j-\tl^j=\tau_j B(u^j-u^{j-1})=-\tau_j^2 d_t B \d_u^j$ and $Bu^j-p^j=d_t \l^j = -d_t \d_\l^j$
we get
\begin{equation}\label{first step conv}
\h{\varrho}_F(p,p^j) + \h{\varrho}_G(u,u^j) \leq -(\d_\l^j,d_t \d_\l^j)_Y - \tau_j^2(d_t B \d_u^j,\d_p^j)_Y.
\end{equation}
Testing the optimality conditions of ~$u^j$ and ~$u^{j-1}$ with $v=u^{j-1}$
and $v=u^j$, respectively, and adding the corresponding inequalities gives
\[
\h{\varrho}_G(u^{j-1},u^j) \leq -\tau_j^2(d_t \l^j,d_t B u^j)_Y.
\]
Using $d_t \l^j = Bu^j-p^j$ and inserting $p=Bu$ on the right-hand side
yields
\begin{equation}\label{second step conv}
\begin{split}
\h{\varrho}_G(u^{j-1},u^j) \leq & \; -\tau_j(B(u^j-u)+p-p^j,B(u^j-u^{j-1}))_Y \\
= & \; -\tau_j^2(B \d_u^j,d_tB\d_u^j)_Y + \tau_j^2(\d_p^j,d_tB\d_u^j)_Y.
\end{split}
\end{equation}
Adding ~\eqref{first step conv} and ~\eqref{second step conv} and using the discrete
product rules ~\eqref{eq:product rule 2} and ~\eqref{eq:product rule 3} gives
\begin{equation}\label{third step conv}
\begin{split}
& \; \h{\varrho}_F(p,p^j) + \h{\varrho}_G(u,u^j) + \h{\varrho}_G(u^{j-1},u^j) + \frac{\tau_j}{2}\|d_t \l^j\|_Y^2 + \frac{\tau_j^3}{2}\|d_t B \d_u^j\|_Y^2 \\
\leq & - \frac{d_t}{2} \|\d_\l^j\|_Y^2 - \tau_j^2 \frac{d_t}{2} \|B \d_u^j\|_Y^2 \\
= & - \frac{d_t}{2} \|\d_\l^j\|_Y^2 - d_t \biggl(\frac{\tau_j^2}{2}\|B\d_u^j\|_Y^2\biggr) + \biggl(d_t \frac{\tau_j^2}{2}\biggr)\|B \d_u^{j-1}\|_Y^2.
\end{split}
\end{equation}

Multiplication by ~$\tau_j$, summation over $j=1,\ldots,J$, and noting that
$R_j^2=\tau_j^2\|d_t \l^j\|_Y^2 + \tau_j^4\|d_t Bu^j\|_Y^2$ yields
\[
\begin{split}
& \frac{1}{2} \Big(\|\d_\l^J\|_Y^2 + \tau_J^2 \|B \d_u^J\|_Y^2 \Big) \\
&+ \sum_{j=1}^J \tau_j  \Bigl(\h{\varrho}_G(u,u^j) + \h{\varrho}_F(p,p^j) + \h{\varrho}_G(u^{j-1},u^j) + \frac{1}{2\tau_j}R_j^2 \Bigr) \\
& \le  \frac{1}{2} \Big(\|\d_\l^0\|_Y^2 + \tau_0^2 \|B\d_u^0\|_Y^2 + \sum_{j=1}^J \tau_j (d_t \tau_j^2) \|B \d_u^{j-1}\|_Y^2 \Big).
\end{split}
\]
The fact that $d_t \tau_j^2 \leq 0$ proves the assertion.
\end{proof}

\begin{remarks}
(1) Note that Algorithm ~\ref{alg:alg_2} is convergent independently of the choice $\tau_0>0$.\\
(2) The estimate shows that a large step size ~$\tau_0$ may affect the convergence behavior.
However, experiments indicate that the algorithm is slow if the step size is chosen too small. This
motivates to consider a variable step size that is adjusted to the performance of the algorithm during
the iteration.\\
(3) If we change the order of minimization in Algorithm ~\ref{alg:alg_2} we obtain the estimate with $\d_p^J$,
$d_t p^j$, $\h{\varrho}_F(p^{j-1},p^j)$, $\d_p^0$ and $\d_p^{j-1}$ instead of $B\d_u^J$,
$d_t B u^j$, $\h{\varrho}_G(u^{j-1},u^j)$, $B\d_u^0$ and $B\d_u^{j-1}$, respectively. The second
minimization should thus be carried out with respect to the variable for which we have strong convexity
to have control over the distance between two consecutive iterates.\\
(4) If ~$X$ and ~$Y$ are finite element spaces related to a triangulation with maximal mesh size ~$h>0$
and if we have $u^0$ with the approximation property
$\|B(u-u^0)\| \leq c h^\a$ we may choose the initial step size ~$\tau_0$ as $\tau_0=h^{-\a}$. In general,
we may initialize the algorithm with a sufficiently large step size ~$\tau_0$ and gradually decrease the
step size, e.g., whenever the algorithm computes iterates which do not lead to a considerable decrease in the residual.\\
(5) Note that the convergence proof allows for finitely many reinitializations of the step size. If
$u^j:=u^0$ and $\l^j:=\l^0$ whenever the step size is reinitialized, this resembles a restart of the
algorithm. To be more precise, if $J_1,\ldots,J_L$ denote the iterations after which the algorithm is
reinitialized, i.e., we set $u^{J_k}:=u^0$ and $\l^{J_k}:=\l^0$ and $\tau_{J_k+1}=\tau_0$, $k=1,\ldots,L$,
we obtain for any $1 \leq k \leq L$ and any $J_k\leq J < J_{k+1}$
\[
\frac{1}{2} D_J^2 + \sum_{j=J_k}^J \Bigl( \tau_j \bigl(\h{\varrho}_G(u,u^j) + \h{\varrho}_F(p,p^j) + \h{\varrho}_G(u^{j-1},u^j)\bigr) + \frac{1}{2}R_j^2\Bigr)
\le  \frac{1}{2} D_0^2,
\]
where we used that $\d_\l^{J_k}=\d_\l^0$ and $B \d_u^{J_k}=B\d_u^0$.
Summation over $k=1,\ldots,L$ then gives for any $J \geq 1$
\[
\frac{L}{2} D_J^2 + \sum_{j=1}^J \Bigl( \tau_j \bigl(\h{\varrho}_G(u,u^j) + \h{\varrho}_F(p,p^j) + \h{\varrho}_G(u^{j-1},u^j)\bigr) + \frac{1}{2}R_j^2\Bigr)
\le  \frac{L}{2} D_0^2.
\]
(6) The conditions on the step size are the same as in ~\cite{km:98}.
\end{remarks}

The residuals ~$R_j$ control the distance between iterates and
a saddle-point $(u,p;\l)$ provided that strong coercivity applies.

\begin{corollary}\label{cor:residual control}
For all $j \geq 1$ we have
\[\begin{split}
\h{\varrho}_F(p,p^j) + \h{\varrho}_G(u,u^j) + \h{\varrho}_G(u^{j-1},u^j) \leq 2C_0 R_j
\end{split}\]
with $C_0=\max\bigl\{\frac{1}{\tau_j}\|\l\|_Y+\|Bu\|_Y, \frac{1}{\tau_j}\|\l^j\|_Y+\|Bu^j\|_Y \bigr\}$.
\end{corollary}

\begin{proof}
Adding ~\eqref{first step conv} and ~\eqref{second step conv} gives
\[\begin{split}
&\h{\varrho}_F(p,p^j) + \h{\varrho}_G(u,u^j) + \h{\varrho}_G(u^{j-1},u^j) \\
\leq & \; - (\d_\l^j,d_t \d_\l^j)_Y - \tau_j^2(B \d_u^j,d_t B \d_u^j)_Y \\
\leq & \; \frac{1}{\tau_j} \|\l-\l^j\|_Y\|\l^j-\l^{j-1}\|_Y + \tau_j\|B(u-u^j)\|_Y\|B(u^j-u^{j-1})\|_Y \\
\leq & \; \Bigl(\frac{1}{\tau_j} \|\l-\l^j\|_Y + \|B(u-u^j)\|_Y\Bigr)R_j,
\end{split}\]
which implies the estimate.
\end{proof}

\begin{remarks}
(1) If ~$G$ is strongly coercive there exists a coercivity
constant $\a_G>0$ such that $\varrho_G(v,w)=\a_G \|v-w\|_X^2$. Particularly, we then have
$\h{\varrho}_G(v,w)=2\a_G \|v-w\|_X^2$.\\
(2) Corollary ~\ref{cor:residual control} motivates to use the stopping criterion
$R_j \leq \veps_{stop}/C_0$ for a prescribed accuracy ~$\veps_{stop}>0$.
\end{remarks}

\subsection{Monotonicity and convergence rate}\label{subsec:monotonicity and convergence rate}
In ~\cite{hy:15} a sublinear $\mathcal{O}(1/J)$ convergence rate for the ADMM is shown
with a contraction-type analysis. For this, the authors prove ~\eqref{third step conv} with constant step sizes
and a monotonicity property of the residual. The residual of Algorithm ~\ref{alg:alg_2} also enjoys
a monotonicity property which is stated in the following proposition which is a generalization of
~\cite[Thm. 5.1]{hy:15}.

\begin{proposition}[Monotonicity of residual]\label{prop:monotone}
For all $j \geq 1$ we have
\[
\begin{split}
2\tau_{j+1}\bigl(&\h{\varrho}_F(p^j,p^{j+1}) + \h{\varrho}_G(u^j,u^{j+1})\bigr) + R_{j+1}^2 \leq R_j^2
\end{split}
\]
if $\tau_{j+1} \leq \tau_j$. Particularly, the residual is non-increasing.
\end{proposition}

\begin{proof}
Testing the decoupled optimality conditions of $(p^{j+1},u^{j+1})$ and $(p^j,u^j)$ in Lemma ~\ref{lemma:dec opt cond}
with $(q,v)=(p^j,u^j)$ and $(q,v)=(p^{j+1},u^{j+1})$, respectively, and adding the inequalities yields
\[
\begin{split}
&\h{\varrho}_F(p^j,p^{j+1}) + \h{\varrho}_G(u^j,u^{j+1}) \\
\leq & \; -(\tl^{j+1}-\tl^j,p^j-p^{j+1})_Y - (\l^j-\l^{j+1},B(u^j-u^{j+1}))_Y \\
= & \; \tau_{j+1}^2 (d_t \tl^{j+1},d_t p^{j+1})_Y - \tau_{j+1}^2 (d_t\l^{j+1},d_t Bu^{j+1})_Y.
\end{split}
\]
Using $d_t \tl^{j+1} = d_t \l^{j+1} - d_t(\tau_{j+1}^2d_tBu^{j+1})$ and
$d_t p^{j+1} = -d_t^2 \l^{j+1} + d_t B u^{j+1}$ in the first term on the right-hand side,
using the discrete product rules ~\eqref{eq:product rule 2}, ~\eqref{eq:product rule 3} 
and Young's inequality $ab \leq \veps a^2/2 + b^2/(2\veps)$ with $\veps=\tau_{j+1}$ we obtain
\[\begin{split}
&\h{\varrho}_F(p^j,p^{j+1}) + \h{\varrho}_G(u^j,u^{j+1}) \\
\leq & \; -\tau_{j+1}^2 (d_t \l^{j+1},d_t^2 \l^{j+1})_Y - (d_t(\tau_{j+1}^2 d_t B u^{j+1}),\tau_{j+1}^2 d_t B u^{j+1})_Y \\
&+ \tau_{j+1}^2 (d_t(\tau_{j+1}^2d_t B u^{j+1}),d_t^2 \l^{j+1})_Y \\
= & \; -\tau_{j+1}^2 \frac{d_t}{2}\|d_t \l^{j+1}\|_Y^2 - \frac{\tau_{j+1}^3}{2}\|d_t^2\l^{j+1}\|_Y^2 \\
&-\frac{d_t}{2}\|\tau_{j+1}^2 d_t B u^{j+1}\|_Y^2 - \frac{\tau_{j+1}}{2}\|d_t (\tau_{j+1}^2d_t B u^{j+1})\|_Y^2 \\
&+ \tau_{j+1}^2 (d_t(\tau_{j+1}^2d_t B u^{j+1}),d_t^2 \l^{j+1})_Y\\
\leq & \; -\frac{d_t}{2}\|\tau_{j+1}^2 d_t B u^{j+1}\|_Y^2 - d_t\Bigl(\frac{\tau_{j+1}^2}{2}\|d_t\l^{j+1}\|_Y^2\Bigr)
+ \Bigl(d_t \frac{\tau_{j+1}^2}{2}\Bigr) \|d_t \l^j\|_Y^2,
\end{split}\]
which implies the assertion.
\end{proof}

We can now deduce a convergence rate for the residual. This generalizes the result
in ~\cite[Thm. 6.1]{hy:15} for the ADMM with monotonically decreasing step sizes.

\begin{corollary}\label{cor:sublinear convergence}
Suppose that the step sizes satisfy the monotonicity property $\tau_{j+1} \leq \tau_j$
for $j \geq 1$. Then we have
\[
R_J^2 = \mathcal{O}\Bigl(\frac{1}{J}\Bigr).
\]
\end{corollary}

\begin{proof}
Proposition ~\ref{prop:monotone} guarantees $R_J \leq R_j$ for $1 \leq j \leq J$ which implies
\[\begin{split}
J R_J^2 \leq \sum_{j=1}^J R_j^2 \leq D_0^2
\end{split}\]
for any $J \geq 1$ where the second inequality is due to Theorem ~\ref{thm:conv}.
\end{proof}

\subsection{Linear convergence}\label{subsec:linear convergence}

We extend the results in ~\cite{dy:16}
concerning the linear convergence of the ADMM to the case
of variable step sizes and prove additionally the linear convergence of the residual
of the Variable-ADMM which serves as the basis for our adjustment rule for the
step size. From now on we assume that
$G$ is Fr\'{e}chet-differentiable with Lipschitz continuous
derivative ~$G'$, i.e., there exists $L_G>0$ such that for all $v,v' \in X$ we have
\[
\|\nabla G(v)-\nabla G(v')\|_X \leq L_G \|v-v'\|_X,
\]
and that ~$G$ is strongly convex with coercivity constant ~$\a_G$, i.e.,
\[
\varrho_G(v,v') \geq \a_G \|v-v'\|_X^2.
\]
Here, ~$\nabla G$ is the representation of ~$G'$ with respect to the inner product $(\cdot,\cdot)_X$,
i.e.,
\[
(\nabla G(v),w)_X = G'(v)[w]
\]
for all $w \in X$. We further assume that the adjoint of ~$B$ is injective with bounded left-inverse, i.e.,
there exists a constant $\a_B>0$ such that $\|B'\mu\|_X \geq \a_B \|\mu\|_Y$ for all
$\mu \in Y$.

\begin{lemma}\label{lemma:strong convexity and co-coercivity}
For all $v,w \in X$ and any $\theta \in [0,1]$ we have
\[\begin{split}
(\nabla G(v)-\nabla G(w),&\;v-w)_X \\
\geq & \; (1-\theta)\frac{1}{L_G} \|\nabla G(v)-\nabla G(w)\|_X^2 + \theta 2 \a_G \|v-w\|_X^2.
\end{split}\]
\end{lemma}

\begin{proof}
Due to the differentiability and strong convexity of ~$G$ we have
\[
(\nabla G(v),w-v)_X + G(v) + \a_G \|w-v\|_X^2 \leq G(w).
\]
Exchanging the roles of ~$v$ and ~$w$ and adding the inequalities gives
\[
(\nabla G(v)-\nabla G(w),v-w)_X \geq 2 \a_G \|v-w\|_X^2.
\]
By Lemma ~\ref{lemma:co-coercivity} stated in Appendix ~\ref{apdx:co-coercivity} we also have
\[
(\nabla G(v)-\nabla G(w),v-w)_X \geq \frac{1}{L_G} \|\nabla G(v)-\nabla G(w)\|_X^2,
\]
which implies the estimate.
\end{proof}

\begin{theorem}[Linear convergence]\label{thm:lin conv}
If $\underline{\tau} \leq \tau_{j+1} \leq \tau_j \leq \overline{\tau}$ for all $j \geq 0$
there exists a sequence $(\g_j)_{j \in \N} \subset (0,1)$ with $\g_j \leq \g <1$ for all $j \geq 1$ such that
\[
D_{j+1}^2 \leq \gamma_{j+1} D_j^2
\]
with ~$D_j$ as in Theorem ~\ref{thm:conv}.
\end{theorem}

\begin{proof}
From inequality ~\eqref{third step conv} it follows that
\begin{equation}\label{eq:lin conv 1}
2\tau_{j+1} \h{\varrho}_G(u,u^{j+1}) + D_{j+1}^2 \leq D_j^2.
\end{equation}
Here, the term ~$\h{\varrho}_G(u,u^{j+1})$ on the left-hand side results from the estimate
\begin{equation}\label{eq:lin conv 2}
\h{\varrho}_G(u,u^{j+1}) \leq (\l^{j+1}-\l,B(u-u^{j+1}))_Y.
\end{equation}
We aim at replacing this bound by a stronger bound using the differentiability of ~$G$ and
Lemma ~\ref{lemma:strong convexity and co-coercivity}.
With the differentiability of ~$G$ the optimality conditions for ~$u$ and ~$u^{j+1}$ in
Lemma ~\ref{lemma:opt cond} and Lemma ~\ref{lemma:dec opt cond}, respectively,
can be rewritten as
\[\begin{split}
-(\l,Bv)_Y &= (\nabla G(u),v)_X, \\
-(\l^{j+1},Bv)_Y &= (\nabla G(u^{j+1}),v)_X
\end{split}\]
for all $v \in X$. Particularly, $\nabla G(u)=-B' \l$ and $\nabla G(u^{j+1})=-B' \l^{j+1}$.
Choosing $v=u-u^{j+1}$ and subtracting we obtain
\[
(\l^{j+1}-\l,B(u-u^{j+1}))_Y = (\nabla G(u)-\nabla G(u^{j+1}),u-u^{j+1})_X.
\]
Using Lemma ~\ref{lemma:strong convexity and co-coercivity} we find with the bounds
$\|Bv\|_Y\leq c_B\|v\|_X$ and $\|B' \mu\|_X \geq \a_B \|\mu\|_Y$ that
\[\begin{split}
(\l^{j+1}-\l,&\;B(u-u^{j+1}))_X \\
\geq & \; (1-\theta)\frac{1}{L_G} \|B'(\l-\l^{j+1})\|_X^2 + \theta 2 \a_G \|u-u^{j+1}\|_X^2 \\
\geq & \; (1-\theta)\frac{\a_B^2}{L_G} \|\d_\l^{j+1}\|_X^2 + \theta \frac{2\a_G}{c_B^2} \|B\d_u^{j+1}\|_X^2 \\
= & \; (1-\theta)s_1\frac{1}{\tau_{j+1}} \|\d_\l^{j+1}\|_X^2 + \theta s_2 \tau_{j+1} \|B\d_u^{j+1}\|_X^2
\end{split}\]
where $s_1=\a_B^2 \tau_{j+1}/L_G$ and $s_2=2 \a_G/(c_B^2 \tau_{j+1})$.
Choosing ~$\theta \in [0,1]$ such that $(1-\theta)s_1=\theta s_2$, i.e.,
\[
\theta = \biggl(1+\frac{s_2}{s_1}\biggr)^{-1},
\]
we have
\[
\rho_{j+1} = (1-\theta)s_1 = \theta s_2 = \frac{s_1s_2}{s_1+s_2} >0.
\]
This choice of ~$\theta$ yields
\[
(\l^{j+1}-\l,B(u-u^{j+1}))_Y \geq \rho_{j+1} \biggl(\tau_{j+1} \|B\d_u^{j+1}\|_X^2 + \frac{1}{\tau_{j+1}}\|\d_\l^{j+1}\|_Y^2\biggr).
\]
Using this estimate instead of ~\eqref{eq:lin conv 2} in the proof of ~\eqref{eq:lin conv 1} we obtain
\[
2\rho_{j+1}D_{j+1}^2 + D_{j+1}^2 \leq D_j^2.
\]
The assertion follows with the choice $\g_{j+1}=(1+2\rho_{j+1})^{-1}$.
\end{proof}

\begin{proposition}[Linear convergence of residual]\label{prop:lin conv residual}
If $\underline{\tau} \leq \tau_{j+1} \leq \tau_j \leq \overline{\tau}$ for all $j \geq 0$
there exists a sequence $(\g_j)_{j \in \N} \subset (0,1)$ with $\g_j \leq \g <1$ for all $j \geq 1$ such that
\[
R_{j+1}^2 \leq \gamma_{j+1} R_j^2.
\]
\end{proposition}

\begin{proof}
We have by Proposition ~\ref{prop:monotone} that
\[
2\tau_{j+1} \h{\varrho}_G(u^j,u^{j+1}) + R_{j+1}^2 \leq R_j^2,
\]
where the term ~$\h{\varrho}_G(u^j,u^{j+1})$ on the left-hand side results from the estimate
\[
\h{\varrho}_G(u^j,u^{j+1}) \leq (\l^{j+1}-\l^j,B(u^j-u^{j+1}))_Y.
\]
The optimality conditions for ~$u^j$ and ~$u^{j+1}$ in Lemma ~\ref{lemma:dec opt cond}
and the differentiability of ~$G$ imply
\[
-B' \l^{j+1} = \nabla G(u^{j+1}), \quad -B' \l^j = \nabla G(u^j).
\]
The assertion follows with the same rate ~$\g_{j+1}$ as in Theorem ~\ref{thm:lin conv} using the same
arguments as in the proof of the previous theorem.
\end{proof}

\begin{remark}\label{rmk:optimal step size}
(1) Note that by the previous proposition and Corollary ~\ref{cor:sublinear convergence} we have
\[
R_J^2 = \mathcal{O}(\min\{\g^J,1/J\}).
\]
(2) Minimizing ~$\g_{j+1}$ with respect to ~$\tau_{j+1}$ yields the step size $\tau_{j+1}\equiv \tau$
with corresponding rate ~$\g_{j+1} \equiv \g$ given by
\[
\tau = \biggl(\frac{2 \a_G L_G}{\a_B^2c_B^2}\biggr)^{1/2}, \quad \g = \biggl(1+\biggl(\frac{2\a_GL_G \a_B^2}{c_B^2}\biggr)^{1/2}\biggr)^{-1}.
\]
\end{remark}

\subsection{Step size adjustment}\label{subsec:adaptivity}
The previous discussion shows that the convergence rate critically depends on the step size.
Moreover, the optimized step size may lead to a pessimistic contraction order as observed in ~\cite{dy:16}.
This motivates to incorporate an automated adjustment of the step size.
With regard to Proposition ~\ref{prop:lin conv residual} the idea is to prescribe a contraction factor
$\g \geq \underline{\g}$, start with a large step size $\tau=\overline{\tau}$ and decrease ~$\tau$
whenever the contraction property is violated. When a lower bound ~$\underline{\tau}$ is reached, the algorithm is restarted
with a larger contraction factor ~$\g$. To account for cases which do not satisfy the conditions for linear
convergence of the algorithm, one has to choose an upper bound ~$\overline{\g} \approx 1$ for the contraction
factor to guarantee convergence of the algorithm. We make our procedure precise in the following algorithm
which is identical to Algorithm ~\ref{alg:alg_2} except for the initialization of additional parameters and the
specification of step ~(7) of Algorithm ~\ref{alg:alg_2}.

\begin{algorithm}[Variable-ADMM]\label{alg:acc_alg_2}
Choose $(u^0,\l^0)\in X\times Y$ such that $G(u^0)<\infty$.
Choose $\overline{\tau} > \underline{\tau}>0$, $\d \in (0,1)$, $0<\underline{\g} \leq \overline{\g}<1$,
$\overline{R} \gg 0$ and $\veps_{stop}>0$. Set $j=1$.\\
(1) Set $\g_1=\underline{\g}$, $\tau_1=\overline{\tau}$ and $R_0=\overline{R}$.\\
(2)-(6) As in Algorithm ~\ref{alg:alg_2}.\\
(7) Define $(\tau_{j+1},\g_{j+1})$ as follows:
\begin{itemize}
\item If $R_j \leq \gamma_j R_{j-1}$ or if $\tau_j=\underline{\tau}$ and $\g_j=\overline{\g}$ set
\[\tau_{j+1}=\tau_j \quad \text{and} \quad \g_{j+1}=\g_j.
\]
\item If $R_j > \gamma_j R_{j-1}$ and $\tau_j>\underline{\tau}$ set
\[
\tau_{j+1}=\max\{\d \tau_j,\underline{\tau}\} \quad \text{and} \quad \g_{j+1}=\g_j.
\]
\item If $R_j > \gamma_j R_{j-1}$, $\tau_j=\underline{\tau}$ and $\g_j < \overline{\g}$ set
\[
\tau_{j+1}=\overline{\tau}, \; \g_{j+1} = \min\Bigl\{\frac{\g_j+1}{2},\overline{\g}\Bigr\}, \; u^j = u^0 \; \text{and} \; \l^j=\l^0.
\]
\end{itemize}
(8) Set $j=j+1$ and continue with ~(2).
\end{algorithm}

\begin{remark}
The total number of restarts is bounded by $\lceil\log((\underline{\g}-1)/(\overline{\g}-1))/\log(2)\rceil$.
The minimal number of iterations between two restarts is given by $\lceil\log(\underline{\tau}/\overline{\tau})/\log(\d)\rceil$.
Since the contraction factor is constant between two restarts, i.e., $\g_j\equiv \widetilde{\g}$ for a
$\widetilde{\g} \in [\underline{\g},\overline{\g})$, the maximal number of iterations between two restarts is bounded by
$\lceil\log(\underline{\tau}/\overline{\tau})/\log(\d)\rceil + \lfloor\log(\veps_{stop}/R_1)/\log(\widetilde{\g})\rfloor$.\\
\end{remark}

\section{Numerical Experiments}\label{sec:numerical experiments}

We tested the ADMM (Algorithm ~\ref{alg:alg_2} with fixed step size), the Variable-ADMM (Algorithm ~\ref{alg:acc_alg_2})
and the Fast-ADMM proposed in ~\cite{gosb:14}, which we present in Appendix ~\ref{apdx:fast-admm} as
Algorithm ~\ref{alg:fast_admm}, for some prototypical
minimization problems which were discretized using low order finite elements. An overview of
relevant parameters and abbreviations is given in the following and concern all three algorithms
if not otherwise stated:
\begin{itemize}
\item \textit{Abbreviations:} $N=\text{\# of iterations for termination}$;
$u_h^{stop}=\text{output}$; $u_h^{ref}=\text{reference solution}$ computed with ADMM,
$\tau=h^{-1}$, $\veps_{stop}=10^{-9}$ (Example ~\ref{ex:obstacle})
and $\tau=h^{-3/2}$, $\veps_{stop}=10^{-4}$ (Example ~\ref{ex:rof})
\item \textit{Geometry:} $\O=(0,1)^2$; coarse mesh ~$\cT_0=\{T_1,T_2\}$;
$\cT_\ell$ generated from ~$\cT_0$ by ~$\ell$ uniform refinements
\item \textit{Mesh sizes:} $h=\sqrt{2}2^{-\ell}$, $\ell=3,\ldots,9$
\item \textit{(Initial) step sizes:} $\overline{\tau}=h^{-m}$, $m=0,\ldots,3$
\item \textit{Initialization:} $u_h^0=0$ and $\l_h^0=0$
\item \textit{Stopping criteria:} (a) $\h{\vrho}_G(u_h^{ref},u_h^j)^{1/2} \leq \veps_{stop}^{(1)}$\\
\hspace*{3.15cm}(b) $R_j \leq \veps_{stop}^{(2)}/C_0$
\item \textit{Stopping tolerance:} $\veps_{stop}^{(1)}=10^{-3}$, $\veps_{stop}^{(2)}=h^2$ (Example ~\ref{ex:obstacle});\\ \hspace*{3.4cm}$\veps_{stop}^{(1)}=10^{-2}$, $\veps_{stop}^{(2)}=h$ (Example ~\ref{ex:rof})
\item \textit{Error:} $E_h=\h{\vrho}_G(u_h^{ref},u_h^{stop})^{1/2}$
\item \textit{Termination:} A hyphen ~($-$) abbreviates ~$N\geq10^3$ (Example ~\ref{ex:obstacle})\\
\hspace*{6.8cm}~$N \geq10^4$ (Example ~\ref{ex:rof})
\item \textit{Fast-ADMM:} $\g=0.999$; $N_{re}=\text{\# of restarts}$ (see step ~(7))
\item \textit{Variable-ADMM:} $\underline{\g}=0.5$, $\d=0.5$, $\underline{\tau}=1$, $\overline{\g}=0.999$;\\
\hspace*{3cm}$N_\tau=\text{\# of $\tau$-adjustments after last restart}$; \\
\hspace*{3cm}$N_\g=\text{\# of $\g$-adjustments}$
\end{itemize}
The number ~$N_\g$ represents also the total number of restarts of the Variable-ADMM when
$\overline{\tau}>\underline{\tau}$. Moreover, the integer ~$N$ also
regards the discarded iterations due to restart in the Fast-ADMM and the Variable-ADMM.

\subsection{Application to obstacle problem}

We consider
\[
G(u_h) = \frac{1}{2} \int_\O |\nabla u_h|^2 \dv{x} - \int_\O f u_h \dv{x}, \quad F(Bu_h) = I_K(u_h)
\]
with $X_\ell = (\cS_0^1(\cT_\ell),(\nabla \cdot,\nabla \cdot))$, $Y_\ell = (\cS_0^1(\cT_\ell),(\cdot,\cdot)_h)$, $B=\operatorname{id}: X \to Y$ and $K$ being the convex set $K=\{v_h \in \cS^1(\cT_\ell): \; v_h \geq \chi\}$ for some obstacle function $\chi \in \cS^1(\cT_\ell)$.

\subsubsection{Example and stopping criterion}

For our experiments we use the following specifications.

\begin{example}\label{ex:obstacle}
We let $f\equiv -5$, $\chi \equiv -1/4$, $\Gamma_D = \partial \O$, $u=0$ on ~$\Gamma_D$.
\end{example}

\begin{table}[!htb]
\centering
\begin{tabular}{|C{0.8cm}||C{0.6cm}|C{1.5cm}||C{0.6cm}|C{1.5cm}||C{0.6cm}|C{1.5cm}||C{0.6cm}|C{1.5cm}|}\firsthline
\multicolumn{1}{|c||}{} & \multicolumn{8}{c|}{ADMM (obstacle; $\veps_{stop}^{(1)}=10^{-3}$, $\tau_j\equiv \overline{\tau}$)} \\ \hline
\multicolumn{1}{|c||}{\rule{0pt}{0.4cm}} & \multicolumn{2}{c||}{$\overline{\tau}=1$} & \multicolumn{2}{c||}{$\overline{\tau}=h^{-1}$} & \multicolumn{2}{c||}{$\overline{\tau}=h^{-2}$} & \multicolumn{2}{c|}{$\overline{\tau}=h^{-3}$}\\ \hline
\multicolumn{1}{|c||}{\rule{0pt}{0.4cm}$\ell$} & \multicolumn{2}{c||}{$N$} & \multicolumn{2}{c||}{$N$} & \multicolumn{2}{c||}{$N$} & \multicolumn{2}{c|}{$N$}\\ \hline
\multicolumn{1}{|c||}{\rule{0pt}{0.4cm}$3$} & \multicolumn{2}{c||}{$880$} & \multicolumn{2}{c||}{$159$} & \multicolumn{2}{c||}{$31$} & \multicolumn{2}{c|}{$30$}\\
\multicolumn{1}{|c||}{4} & \multicolumn{2}{c||}{$-$} & \multicolumn{2}{c||}{$221$} & \multicolumn{2}{c||}{$28$} & \multicolumn{2}{c|}{$217$}\\
\multicolumn{1}{|c||}{5} & \multicolumn{2}{c||}{$-$} & \multicolumn{2}{c||}{$241$} & \multicolumn{2}{c||}{$76$} & \multicolumn{2}{c|}{$-$}\\
\multicolumn{1}{|c||}{6} & \multicolumn{2}{c||}{$-$} & \multicolumn{2}{c||}{$151$} & \multicolumn{2}{c||}{$294$} & \multicolumn{2}{c|}{$-$}\\
\multicolumn{1}{|c||}{7} & \multicolumn{2}{c||}{$-$} & \multicolumn{2}{c||}{$103$} & \multicolumn{2}{c||}{$-$} & \multicolumn{2}{c|}{$-$}\\
\multicolumn{1}{|c||}{8} & \multicolumn{2}{c||}{$-$} & \multicolumn{2}{c||}{$59$} & \multicolumn{2}{c||}{$-$} & \multicolumn{2}{c|}{$-$}\\
\multicolumn{1}{|c||}{9} & \multicolumn{2}{c||}{$-$} & \multicolumn{2}{c||}{$60$} & \multicolumn{2}{c||}{$-$} & \multicolumn{2}{c|}{$-$}\\ \hline
\multicolumn{1}{|c||}{\rule{0pt}{0.4cm}} & \multicolumn{8}{c|}{Fast-ADMM ($\g=0.999$)} \\ \hline
\multicolumn{1}{|c||}{\rule{0pt}{0.4cm}$\ell$} & $N$ & $(N_{re})$ & $N$ & $(N_{re})$ & $N$ & $(N_{re})$ & $N$ & $(N_{re})$ \\ \hline
\multicolumn{1}{|c||}{\rule{0pt}{0.4cm}$3$}  & $  123$ & $(   3)$ & $   45$ & $(   3)$ & $   17$ & $(   1)$ & $   19$ & $(   2)$\\
$4$ & $  241$ & $(   6)$ & $   59$ & $(   2)$ & $   22$ & $(   1)$ & $   51$ & $(   2)$\\
$5$ & $  474$ & $(  17)$ & $   67$ & $(   4)$ & $   34$ & $(   2)$ & $  157$ & $(   2)$\\
$6$ & $ -$ & $( -)$ & $   47$ & $(   2)$ & $   58$ & $(   2)$ & $ -$ & $( -)$\\
$7$ & $ -$ & $( -)$ & $   62$ & $(   6)$ & $  123$ & $(   2)$ & $ -$ & $(-)$\\
$8$ & $ -$ & $( -)$ & $   37$ & $(   4)$ & $  320$ & $(  32)$ & $ -$ & $( -)$\\
$9$ & $ -$ & $( -)$ & $   32$ & $(   2)$ & $ -$ & $( -)$ & $ -$ & $(-)$\\ \hline
\multicolumn{1}{|c||}{\rule{0pt}{0.4cm}} & \multicolumn{8}{c|}{Variable-ADMM ($\underline{\tau}=1$, $\underline{\g}=0.5$, $\overline{\g}=0.999$, $\d=0.5$)} \\ \hline
\multicolumn{1}{|c||}{\rule{0pt}{0.4cm}$\ell$} & $N$ & $(N_\tau,N_\g)$ & $N$ & $(N_\tau,N_\g)$ & $N$ & $(N_\tau,N_\g)$ & $N$ & $(N_\tau,N_\g)$ \\ \hline
\multicolumn{1}{|c||}{\rule{0pt}{0.4cm}$3$}  & $  880$ & $( 0, 8)$ & $  236$ & $( 0, 5)$ & $   83$ & $( 0, 3)$ & $   59$ & $( 1, 2)$\\
$4$ & $ -$ & $( -, -)$ & $  406$ & $( 0, 6)$ & $   60$ & $( 0, 2)$ & $  116$ & $( 3, 3)$\\
$5$ & $ -$ & $( -, -)$ & $  547$ & $( 0, 7)$ & $  123$ & $( 1, 3)$ & $  213$ & $( 5, 4)$\\
$6$ & $ -$ & $( -, -)$ & $  442$ & $( 0, 7)$ & $  134$ & $( 3, 3)$ & $  233$ & $( 8, 4)$\\
$7$ & $ -$ & $( -, -)$ & $  410$ & $( 0, 7)$ & $  151$ & $( 8, 3)$ & $  247$ & $(11, 4)$\\
$8$ & $ -$ & $( -,-)$ & $  237$ & $( 0, 5)$ & $  254$ & $( 6, 4)$ & $  263$ & $(14, 4)$\\
$9$ & $ -$ & $( -, -)$ & $  185$ & $( 0, 4)$ & $  264$ & $( 8, 4)$ & $  277$ & $(17, 4)$\\ \lasthline
\end{tabular}
\caption{Iteration numbers for Example ~\ref{ex:obstacle} using ADMM, Fast-ADMM and Variable-ADMM with stopping criterion
$\|\nabla (u_h^{ref}-u_h^j)\| \leq 10^{-3}$. A hyphen ~($-$) means that the algorithm did not terminate within ~$10^3$ iterations. In parenthesis: total number of restarts (Fast-ADMM) and total number of adjustments of ~$\tau_j$ and ~$\g_j$ (Variable-ADMM).}\label{tab:iteration numbers for obstacle error 1e-3}
\end{table}

\begin{table}[!htb]
\centering
\begin{tabular}{|C{0.8cm}||C{0.6cm}|C{1.27cm}||C{0.6cm}|C{1.27cm}||C{0.6cm}|C{1.27cm}||C{0.6cm}|C{1.27cm}|}\firsthline
\multicolumn{1}{|c||}{} & \multicolumn{8}{c|}{ADMM (obstacle; $\veps_{stop}^{(2)}=h^2$, $\tau_j\equiv \overline{\tau}$)} \\[1pt] \hline
\multicolumn{1}{|c||}{\rule{0pt}{0.4cm}} & \multicolumn{2}{c||}{$\overline{\tau}=1$} & \multicolumn{2}{c||}{$\overline{\tau}=h^{-1}$} & \multicolumn{2}{c||}{$\overline{\tau}=h^{-2}$} & \multicolumn{2}{c|}{$\overline{\tau}=h^{-3}$}\\[1pt] \hline
\multicolumn{1}{|c||}{\rule{0pt}{0.4cm}$\ell$} & $N$ & $E_h/h$ & $N$ & $E_h/h$ & $N$ & $E_h/h$ & $N$ & $E_h/h$ \\[1pt] \hline
\multicolumn{1}{|c||}{\rule{0pt}{0.4cm}$3$} & $    4$ & $0.8658$ & $   14$ & $0.1659$ & $    9$ & $0.0523$ & $   26$ & $0.0160$\\
$4$ & $   62$ & $0.5479$ & $   35$ & $0.1401$ & $   25$ & $0.0160$ & $  218$ & $0.0106$\\
$5$ & $  221$ & $0.3439$ & $   57$ & $0.0819$ & $   87$ & $0.0050$ & $ -$ & $-$\\
$6$ & $  697$ & $0.2868$ & $  197$ & $0.0357$ & $  382$ & $0.0026$ & $ -$ & $-$\\
$7$ & $ -$ & $-$ & $  538$ & $0.0221$ & $ -$ & $-$ & $-$ & $-$\\
$8$ & $ -$ & $-$ & $ -$ & $-$ & $ -$ & $-$ & $ -$ & $-$\\
$9$ & $ -$ & $-$ & $ -$ & $-$ & $-$ & $-$ & $ -$ & $-$\\ \hline
\multicolumn{1}{|c||}{\rule{0pt}{0.4cm}$\ell$} & \multicolumn{8}{c|}{Fast-ADMM ($\g=0.999$)} \\[1pt] \hline
\multicolumn{1}{|c||}{\rule{0pt}{0.4cm}$3$} & $    4$ & $0.8455$ & $   10$ & $0.1723$ & $    9$ & $0.0490$ & $   15$ & $0.0081$\\
$4$ & $   17$ & $0.5550$ & $   18$ & $0.1395$ & $   19$ & $0.0161$ & $   51$ & $0.0089$\\
$5$ & $  104$ & $0.3442$ & $   39$ & $0.0799$ & $   36$ & $0.0044$ & $  184$ & $0.0043$\\
$6$ & $  543$ & $0.2870$ & $   56$ & $0.0368$ & $   71$ & $0.0014$ & $-$ & $-$\\
$7$ & $ -$ & $-$ & $  185$ & $0.0220$ & $  168$ & $0.0010$ & $ -$ & $-$\\
$8$ & $ -$ & $-$ & $  462$ & $0.0124$ & $  466$ & $0.0004$ & $-$ & $-$\\
$9$ & $-$ & $-$ & $-$ & $-$ & $ -$ & $-$ & $-$ & $-$\\ \hline
\multicolumn{1}{|c||}{\rule{0pt}{0.4cm}$\ell$} & \multicolumn{8}{c|}{Variable-ADMM ($\underline{\tau}=1$, $\underline{\g}=0.5$, $\overline{\g}=0.999$, $\d=0.5$)} \\[1pt] \hline
\multicolumn{1}{|c||}{\rule{0pt}{0.4cm}$3$} & $    4$ & $0.8658$ & $    7$ & $0.3332$ & $    8$ & $0.1586$ & $   10$ & $0.1474$\\
$4$ & $   62$ & $0.5479$ & $   36$ & $0.3808$ & $   29$ & $0.1734$ & $   37$ & $0.1228$\\
$5$ & $  221$ & $0.3439$ & $   68$ & $0.2260$ & $   36$ & $0.2848$ & $   47$ & $0.1877$\\
$6$ & $  697$ & $0.2868$ & $   85$ & $0.2750$ & $   78$ & $0.1189$ & $   96$ & $0.1727$\\
$7$ & $-$ & $-$ & $  145$ & $0.2175$ & $  154$ & $0.0898$ & $  168$ & $0.1433$\\
$8$ & $ -$ & $-$ & $  251$ & $0.1708$ & $  166$ & $0.1914$ & $  290$ & $0.1090$\\
$9$ & $ -$ & $-$ & $  327$ & $0.1557$ & $  304$ & $0.1359$ & $  500$ & $0.0775$\\ \lasthline
\end{tabular}
\caption{Iteration numbers and ratio $E_h/h$ for Example ~\ref{ex:obstacle} using ADMM, Fast-ADMM and Variable-ADMM with $R_j \leq h^2/\til{C}_0$. A hyphen ~($-$) means that the algorithm did not terminate within ~$10^3$ iterations.}\label{tab:iteration numbers for obstacle residual h2}
\end{table}

Note that ~$G$ is strongly convex with coercivity constant $\a_G=1/2$, i.e., we have $\h{\varrho}_G(v,w)=\|v-w\|_X^2=\|\nabla(v-w)\|^2$. We have that the unique solution of the infinite-dimensional obstacle problem satisfies $u \in H^2(\O)$, cf. ~\cite{bs:68,sbartels:15}.
Hence, Corollary ~\ref{cor:residual control} implies that the error tolerance should be chosen as $\veps_{stop}=h^2$ so that the convergence rate $\|\nabla(u-u_h)\|=\mathcal{O}(h)$ is not violated. Particularly, ~$G$ is differentiable and its gradient is Lipschitz continuous and ~$B'$ is invertible, i.e., the conditions for linear convergence of the ADMM are satisfied.

With the chosen inner-products on ~$X$ and ~$Y$ we obtain the constants
\[
\a_G=\frac{1}{2}, \quad L_G =1, \quad \a_B \approx \frac{h}{2}, \quad c_B = 2c_P
\]
with ~$c_P$ denoting the Poincar\'{e} constant associated to the domain ~$\O$, which can in turn be
bounded by $c_P \leq \sqrt{2}/\pi$. Using these constants in Remark ~\ref{rmk:optimal step size} leads to $\tau_{opt}\approx \pi/(\sqrt{2}h)$ and $\g_{opt}\approx (1+\pi h/\sqrt{32})^{-1}$.

With regard to Corollary ~\ref{cor:residual control} we have to provide a computable
upper bound ~$\widetilde{C}_0$ for
\[
C_0=\max\Bigl\{\frac{1}{\tau_j}\|\l_h\|_h+\|u_h\|_h, \frac{1}{\tau_j}\|\l_h^j\|_h+\|u_h^j\|_h \Bigr\}.
\]
We have $\|u_h\|_h \leq 2 \|u_h\| \leq c(\|\chi\|+\|f\|)$. Furthermore, the optimality condition
$0 \in G'(u_h) + \partial F(u_h)$ implies the existence of a Lagrange multiplier ~$\l_h \in \partial F(u_h)$ with
\[
-(\l_h,v_h)_h = (\nabla u_h,\nabla v_h) - (f,v_h) \quad \text{for all } v_h \in \cS_0^1(\cT_\ell).
\]
Particularly, Assumption ~\ref{ass:existence of sp} is satisfied for Example ~\ref{ex:obstacle}.
Inserting $(\nabla u,\nabla v_h)$ on the right-hand side, using standard interpolation estimates,
an inverse estimate, the fact that $u \in H^2(\O)$ and integration by parts gives
\[\begin{split}
|(\l_h,v_h)_h| &\leq \|\nabla (u-u_h)\|\|\nabla v_h\| + \|f\|\|v_h\| + \|\Delta u\|\|v_h\| \\
&\leq c (1+\|f\|+\|\Delta u\|) \|v_h\| 
\end{split}\]
which means that $\|\l_h\|_h$ is uniformly bounded in ~$h$.
Therefore, since we only consider $\tau_j \geq \underline{\tau}=1$ in our numerical experiments we
set
\[
\widetilde{C}_0=\max\Bigl\{1,\frac{1}{\tau_j}\|\l_h^j\|_h+\|u_h^j\|_h\Bigr\},
\]
which is, up to constants, an upper bound for ~$C_0$.

\subsubsection{Results}

We report the iteration numbers for the ADMM, the Fast-ADMM and the Variable-ADMM
applied to Example ~\ref{ex:obstacle} with stopping criterion $\|\nabla (u_h^{ref} - u^j)\| \leq 10^{-3}$
in Table ~\ref{tab:iteration numbers for obstacle error 1e-3}. Note that
$\veps_{stop}^{(1)}=10^{-3}$ is a lower bound for the minimal mesh size we are considering in our experiments,
i.e., the outputs of the algorithms do not affect the order of convergence $\|\nabla(u-u_h)\|=\mathcal{O}(h)$.
We infer that, for large initial step sizes, the iteration numbers of the Variable-ADMM are considerably
smaller than those of the ADMM and also smaller or at least comparable to those of the Fast-ADMM. Particularly,
one can observe a mesh-independent convergence behavior for Variable-ADMM. Note that $\overline{\tau}=h^{-1}$
happens to be approximately the optimal step size ~$\tau_{opt}$ which explains the lower iteration numbers
for ADMM and Fast-ADMM in the case $\overline{\tau}=h^{-1}$ since Variable-ADMM had to restart several times
to recognize the actual contraction order.

In Table ~\ref{tab:iteration numbers for obstacle residual h2} the iteration numbers and the
ratio $E_h/h$, which identifies the quality of the stopping criterion, for the three algorithms with stopping criterion $R_j \leq h^2/C_0$
are displayed which also reflect a considerable improvement of the Variable-ADMM over the ADMM and
Fast-ADMM especially for large initial step sizes $\overline{\tau}=h^{-2},h^{-3}$. The ADMM and Variable-ADMM
do not differ for $\overline{\tau}=1$ since we have set $\underline{\tau}=1$. A remarkable feature of the
Variable-ADMM is that it performs robustly with respect to the choice of the initial step size ~$\overline{\tau}$ in contrast to
the ADMM and Fast-ADMM. Let us finally remark that the ratio $E_h/h$ remains bounded as $h \to 0$
which underlines that $R_j \leq \veps_{stop}/C_0$ is a reliable stopping criterion and optimal for Variable-ADMM.

\subsection{Application to $TV$-$L^2$ minimization}

In this subsection we apply the algorithms to a prototypical total variation minimization
problem, the so called \emph{ROF problem}, cf. ~\cite{rof:92}. We set
\[
G(u_h) = \frac{\a}{2} \|u_h-g\|^2, \quad F(Bu_h) = \int_\O |\nabla u_h| \dv{x}
\]
with $X_\ell = (\cS^1(\cT_\ell),(\cdot,\cdot))$, $Y_\ell = (\cL^0(\cT_\ell)^d,(\cdot,\cdot)_w)$, and $B=\nabla: X_\ell \to Y_\ell$.

\subsubsection{Example and stopping criterion}

We consider the following specification of the minimization problem.

\begin{example}\label{ex:rof}
We let $\a=20$ and $g=\til{g}+\xi \in \cS^1(\cT_\ell)$ where $\til{g} \in \cS^1(\cT_\ell)$ is the piecewise linear approximation of the characteristic function ~$\chi_{B_{1/5}(x_\O)}$ of the circle with radius $r=1/5$ around the center ~$x_\O$ of ~$\O$ and $\xi \in \cS^1(\cT_3)$ is a perturbation function whose coefficients are samples of a uniformly distributed random variable in the interval $[-1/10,1/10]$.
\end{example}

\begin{table}[!htb]
\centering
\begin{tabular}{|C{0.8cm}||C{0.77cm}|C{1.5cm}||C{0.77cm}|C{1.5cm}||C{0.77cm}|C{1.5cm}||C{0.77cm}|C{1.5cm}|}\firsthline
\multicolumn{1}{|c||}{} & \multicolumn{8}{c|}{ADMM (ROF; $\veps_{stop}^{(1)}=10^{-2}$, $\tau_j\equiv \overline{\tau}$)} \\ \hline
\multicolumn{1}{|c||}{\rule{0pt}{0.4cm}} & \multicolumn{2}{c||}{$\overline{\tau}=1$} & \multicolumn{2}{c||}{$\overline{\tau}=h^{-1}$} & \multicolumn{2}{c||}{$\overline{\tau}=h^{-2}$} & \multicolumn{2}{c|}{$\overline{\tau}=h^{-3}$}\\ \hline
\multicolumn{1}{|c||}{\rule{0pt}{0.4cm}$\ell$} & \multicolumn{2}{c||}{$N$} & \multicolumn{2}{c||}{$N$} & \multicolumn{2}{c||}{$N$} & \multicolumn{2}{c|}{$N$}\\ \hline
\multicolumn{1}{|c||}{\rule{0pt}{0.4cm}$3$} & \multicolumn{2}{c||}{$195$} & \multicolumn{2}{c||}{$36$} & \multicolumn{2}{c||}{$21$} & \multicolumn{2}{c|}{$127$}\\
\multicolumn{1}{|c||}{4} & \multicolumn{2}{c||}{$647$} & \multicolumn{2}{c||}{$58$} & \multicolumn{2}{c||}{$56$} & \multicolumn{2}{c|}{$623$}\\
\multicolumn{1}{|c||}{5} & \multicolumn{2}{c||}{$2443$} & \multicolumn{2}{c||}{$109$} & \multicolumn{2}{c||}{$178$} & \multicolumn{2}{c|}{$3994$}\\
\multicolumn{1}{|c||}{6} & \multicolumn{2}{c||}{$9827$} & \multicolumn{2}{c||}{$218$} & \multicolumn{2}{c||}{$578$} & \multicolumn{2}{c|}{$-$}\\
\multicolumn{1}{|c||}{7} & \multicolumn{2}{c||}{$-$} & \multicolumn{2}{c||}{$427$} & \multicolumn{2}{c||}{$810$} & \multicolumn{2}{c|}{$-$}\\
\multicolumn{1}{|c||}{8} & \multicolumn{2}{c||}{$-$} & \multicolumn{2}{c||}{$842$} & \multicolumn{2}{c||}{$1430$} & \multicolumn{2}{c|}{$-$}\\
\multicolumn{1}{|c||}{9} & \multicolumn{2}{c||}{$-$} & \multicolumn{2}{c||}{$1670$} & \multicolumn{2}{c||}{$3666$} & \multicolumn{2}{c|}{$-$}\\ \hline
\multicolumn{1}{|c||}{\rule{0pt}{0.4cm}} & \multicolumn{8}{c|}{Fast-ADMM ($\g=0.999$)} \\ \hline
\multicolumn{1}{|c||}{\rule{0pt}{0.4cm}$\ell$} & $N$ & $(N_{re})$ & $N$ & $(N_{re})$ & $N$ & $(N_{re})$ & $N$ & $(N_{re})$ \\[1pt] \hline
\multicolumn{1}{|c||}{\rule{0pt}{0.4cm}$3$} & $   60$ & $(   2)$ & $   19$ & $(   1)$ & $    9$ & $(   0)$ & $   77$ & $(   8)$\\
$4$ & $  174$ & $(   9)$ & $   31$ & $(   2)$ & $   25$ & $(   2)$ & $  386$ & $(  42)$\\
$5$ & $ 1596$ & $( 560)$ & $   50$ & $(   4)$ & $   88$ & $(   7)$ & $ 6410$ & $(2907)$\\
$6$ & $-$ & $(-)$ & $  113$ & $(   9)$ & $  289$ & $(  24)$ & $-$ & $(-)$\\
$7$ & $-$ & $( -)$ & $  226$ & $(  19)$ & $  410$ & $(  33)$ & $-$ & $( -)$\\
$8$ & $-$ & $( -)$ & $  485$ & $(  46)$ & $  790$ & $(  72)$ & $-$ & $( -)$\\
$9$ & $-$ & $( -)$ & $ 1042$ & $( 111)$ & $ 5093$ & $(2158)$ & $-$ & $( -)$\\ \hline
\multicolumn{1}{|c||}{\rule{0pt}{0.4cm}} & \multicolumn{8}{c|}{Variable-ADMM ($\underline{\tau}=1$, $\underline{\g}=0.5$, $\overline{\g}=0.999$, $\d=0.5$)} \\ \hline
\multicolumn{1}{|c||}{\rule{0pt}{0.4cm}$\ell$}& $N$ & $(N_\tau,N_\g)$ & $N$ & $(N_\tau,N_\g)$ & $N$ & $(N_\tau,N_\g)$ & $N$ & $(N_\tau,N_\g)$ \\[1pt] \hline
\multicolumn{1}{|c||}{\rule{0pt}{0.4cm}$3$} & $  195$ & $( 0, 6)$ & $  111$ & $( 0, 4)$ & $   21$ & $( 2, 1)$ & $   25$ & $( 5, 1)$\\
$4$ & $  647$ & $( 0, 8)$ & $  183$ & $( 0, 5)$ & $   50$ & $( 3, 2)$ & $   62$ & $( 6, 2)$\\
$5$ & $ 2443$ & $( 0, 9)$ & $  319$ & $( 0, 6)$ & $   62$ & $( 5, 2)$ & $   82$ & $(10, 2)$\\
$6$ & $ 9827$ & $( 0, 9)$ & $  594$ & $( 0, 7)$ & $  129$ & $( 5, 3)$ & $  166$ & $(10, 3)$\\
$7$ & $-$ & $( -, -)$ & $ 1136$ & $( 0, 8)$ & $  140$ & $( 7, 3)$ & $  180$ & $(14, 3)$\\
$8$ & $-$ & $( -, -)$ & $ 2194$ & $( 0, 9)$ & $  156$ & $( 8, 3)$ & $  197$ & $(14, 3)$\\
$9$ & $-$ & $( -, -)$ & $-$ & $( -, -)$ & $  326$ & $( 7, 4)$ & $  401$ & $(16, 4)$\\ \lasthline
\end{tabular}
\caption{Iteration numbers for Example ~\ref{ex:rof} using ADMM, Fast-ADMM and Variable-ADMM with stopping criterion $\sqrt{\a}\|u_h^{ref}-u_h^j\|\leq 10^{-2}$. A hyphen ~($-$) means that the algorithm did not converge within ~$10^4$ iterations. In parenthesis: total number of restarts (Fast-ADMM) and total number of adjustments of ~$\tau_j$ and ~$\g_j$ (Variable-ADMM).}\label{tab:iteration numbers for rof error 1e-2}
\end{table}

\begin{table}[!htb]
\centering
\begin{tabular}{|C{0.8cm}||C{0.77cm}|C{1.21cm}||C{0.77cm}|C{1.21cm}||C{0.77cm}|C{1.21cm}||C{0.77cm}|C{1.21cm}|}\firsthline
\multicolumn{1}{|c||}{} & \multicolumn{8}{c|}{ADMM (ROF; $\veps_{stop}^{(2)}=h$, $\tau_j\equiv \overline{\tau}$)} \\ \hline
\multicolumn{1}{|c||}{\rule{0pt}{0.4cm}} & \multicolumn{2}{c||}{$\overline{\tau}=1$} & \multicolumn{2}{c||}{$\overline{\tau}=h^{-1}$} & \multicolumn{2}{c||}{$\overline{\tau}=h^{-2}$} & \multicolumn{2}{c|}{$\overline{\tau}=h^{-3}$}\\ \hline
\multicolumn{1}{|c||}{\rule{0pt}{0.4cm}$\ell$} & $N$ & $E_h/\sqrt{h}$ & $N$ & $E_h/\sqrt{h}$ & $N$ & $E_h/\sqrt{h}$ & $N$ & $E_h/\sqrt{h}$ \\ \hline
\multicolumn{1}{|c||}{\rule{0pt}{0.4cm}$3$} & $   27$ & 0.2242 & $    7$ & 0.1531 & $    8$ & 0.1479 & $   49$ & 0.1836\\
$4$ & $  137$ & 0.1690 & $   14$ & 0.1455 & $   47$ & 0.0558 & $  521$ & 0.0569\\
$5$ & $  691$ & 0.1657 & $   32$ & 0.1614 & $  217$ & 0.0248 & $ 4878$ & 0.0251\\
$6$ & $ 3882$ & $0.1739$ & $   87$ & $0.1729$ & $  740$ & $0.0529$ & $-$ & $-$\\
$7$ & $-$ & $-$ & $  286$ & $0.1569$ & $ 2037$ & $0.0169$ & $-$ & $-$\\
$8$ & $-$ & $-$ & $  920$ & $0.1169$ & $ 7598$ & $0.0131$ & $-$ & $-$\\
$9$ & $-$ & $-$ & $ 2585$ & $0.0834$ & $-$ & $-$ & $-$ & $-$\\ \hline
\multicolumn{1}{|c||}{\rule{0pt}{0.4cm}$\ell$} & \multicolumn{8}{c|}{Fast-ADMM ($\g=0.999$)} \\ \hline
\multicolumn{1}{|c||}{\rule{0pt}{0.4cm}$3$} & $   15$ & 0.1188 & $    6$ & 0.1013 & $    7$ & 0.0757 & $   26$ & 0.1625\\
$4$ & $   63$ & 0.1644 & $   13$ & 0.1080 & $   26$ & 0.0342 & $  326$ & 0.0554\\
$5$ & $  246$ & 0.1662 & $   28$ & 0.0975 & $  107$ & 0.0217 & $ 8178$ & 0.0251\\
$6$ & $ 2370$ & $0.1746$ & $   49$ & $0.1693$ & $  384$ & $0.0526$ & $-$ & $-$\\
$7$ & $-$ & $-$ & $  153$ & $0.1561$ & $ 1145$ & $0.0168$ & $-$ & $-$\\
$8$ & $-$ & $-$ & $  537$ & $0.1154$ & $-$ & $-$ & $-$ & $-$\\
$9$ & $-$ & $-$ & $ 1669$ & $0.0829$ & $-$ & $-$ & $-$ & $-$\\ \hline
\multicolumn{1}{|c||}{\rule{0pt}{0.4cm}$\ell$} & \multicolumn{8}{c|}{Variable-ADMM ($\underline{\tau}=1$, $\underline{\g}=0.5$, $\overline{\g}=0.999$, $\d=0.5$)} \\ \hline
\multicolumn{1}{|c||}{\rule{0pt}{0.4cm}$3$} & $   27$ & 0.2242 & $   14$ & 0.1531 & $    6$ & 0.1244 & $   20$ & 0.1310\\
$4$ & $  137$ & 0.1690 & $   34$ & 0.1455 & $   26$ & 0.0957 & $   33$ & 0.1006\\
$5$ & $  691$ & 0.1657 & $   78$ & 0.1614 & $   60$ & 0.0568 & $   81$ & 0.0529\\
$6$ & $ 3882$ & $0.1739$ & $  331$ & $0.1729$ & $  137$ & $0.0426$ & $  181$ & $0.0178$\\
$7$ & $-$ & $-$ & $  995$ & $0.1569$ & $  319$ & $0.0156$ & $  375$ & $0.0162$\\
$8$ & $-$ & $-$ & $ 2272$ & $0.1169$ & $  365$ & $0.0161$ & $  430$ & $0.0202$\\
$9$ & $-$ & $-$ & $-$ & $-$ & $  834$ & $0.0140$ & $ 1763$ & $0.0080$\\ \lasthline
\end{tabular}
\caption{Iteration numbers for Example ~\ref{ex:rof} using ADMM, Fast-ADMM and Variable-ADMM with $R_j \leq h/\til{C}_0$. A hyphen ~($-$) means that the algorithm did not converge within ~$10^4$ iterations. In parenthesis: total number of restarts (Fast-ADMM) and total number of adjustments of ~$\tau_j$ and ~$\g_j$ (Variable-ADMM).}\label{tab:iteration numbers for rof residual h}
\end{table}

Note that ~$G$ is strongly convex with coercivity constant $\a_G=\a/2$, i.e., we have $\h{\varrho}_G(v,w)=\a\|v-w\|^2$. Let $u \in BV(\O)\cap L^2(\O)$ denote the continuous minimizer. Since $g \in L^{\infty}(\O)$ one can show that $u \in L^{\infty}(\O)$. Then by Corollary ~\ref{cor:residual control} the error tolerance has to be chosen as $\veps_{stop}=h$ to match the optimal convergence rate $\|u-u_h\|=\mathcal{O}(h^{1/2})$, cf. ~\cite[Rmk. 7.2]{bns:14} and ~\cite[Rmk. 10.9 (ii)]{sbartels:15}.

The optimality condition $0 \in G'(u_h) + \partial F(\nabla u_h)$ implies the existence of a Lagrange multiplier
$\l_h \in \partial F(\nabla u_h)$ with $\diver \l_h=\a (u_h-g)$ (cf. ~\cite[Thm. 23.9]{rockafellar:70}) where the operator $\diver: \cL^0(\cT_\ell)^d \to \cS^1(\cT_\ell)$ is defined via $-(\diver \mu_h,v_h)=(\mu,\nabla v_h)_w$ for all $\mu_h \in \cL^0(\cT_\ell)^d$ and 
$v_h \in \cS^1(\cT_\ell)$. Hence, Assumption ~\ref{ass:existence of sp} is satisfied. In this setting the constant ~$C_0$ from Corollary ~\ref{cor:residual control} is given by
\[
C_0=\max\Bigl\{\frac{1}{\tau_j}\|\l_h\|_w+\|\nabla u_h\|_w, \frac{1}{\tau_j}\|\l_h^j\|_w+\|\nabla u_h^j\|_w \Bigr\}.
\]
The specific choice of the norm ensures that by an inverse estimate it holds $\|\nabla u_h\|_w \leq c \|\nabla u_h\|_{L^1(\O)}$. The optimality condition $\l_h \in \partial F(p_h)$ implies that
\[
\l_h=\begin{cases} \displaystyle h^{-d} \frac{p_h}{|p_h|}, &\text{ if } p_h \neq 0, \\
h^{-d} \xi, &\text{ if } p_h=0, \end{cases}
\]
with $\xi \in B_1(0)$. Therefore we have $\|\l_h\|_w \leq c h^{-d/2}$. This scaling of the Lagrange multiplier has to be taken into
account in the tolerance for the residual to obtain meaningful outputs, i.e., we set
\[
\widetilde{C}_0 = \max\Bigl\{\frac{1}{h^{d/2}\tau_j},\frac{1}{\tau_j}\|\l_h^j\|_w+\|\nabla u_h^j\|_w\Bigr\}.
\]

\subsubsection{Results}

The iteration numbers for the ADMM, the Fast-ADMM and the Variable-ADMM applied to Example ~\ref{ex:rof} with stopping criterion $\sqrt{\a}\|u_h^{ref}-u^j\| \leq 10^{-2}$ are displayed in Table ~\ref{tab:iteration numbers for rof error 1e-2}. Note that $\veps_{stop}^{(1)}=10^{-2}$ is a lower bound for ~$h^{1/2}$, i.e., the optimal convergence rate $\|u-u_h\|=\mathcal{O}(h^{1/2})$ is not affected by the computed outputs. The results again underline that Variable-ADMM can lead to a considerable improvement for large initial step sizes even though this example does not
fit into the framework of linear convergence of ADMM we addressed in Subsection ~\ref{subsec:linear convergence}.

In Table ~\ref{tab:iteration numbers for rof residual h} the iteration numbers with the stopping criterion $R_j \leq h/C_0$ are reported. Here, the advantage of using the Variable-ADMM is even more pronounced than in Example ~\ref{ex:obstacle} since again the iteration numbers especially for the initial step sizes $\overline{\tau}=h^{-2}, h^{-3}$ are lower compared to those of ADMM and Fast-ADMM.
The reported ratio $E_h/\sqrt{h}$ once more confirms the reliability of the chosen stopping criterion.

\subsection{Failure of other stopping criteria}

We next demonstrate that it is not sufficient to stop the ADMM and Fast-ADMM using
as the stopping criterion either $\|\l_h^j-\l_h^{j-1}\|_Y \leq \veps_{stop}$ or $\tau_j\|B(u_h^j-u_h^{j-1})\|_Y\leq \veps_{stop}$
because these stopping criteria may not always
lead to a suitable approximation of the exact minimizer ~$u_h$ and that one needs to resort to the stronger
stopping criterion $R_j \leq \veps_{stop}/C_0$ due to the saddle-point structure. To see this, we consider Example
~\ref{ex:obstacle} with the stopping criterion $\|\l_h^j-\l_h^{j-1}\|_h \leq h^2/C_0$
and Example ~\ref{ex:rof} with stopping criterion $\tau_j\|\nabla(u_h^j-u_h^{j-1})\|_w \leq h/C_0$ and investigate the
ratio $\|\nabla (u_h^{ref}-u_h^{stop})\|/h$ and $\sqrt{\a}\|u_h^{ref}-u_h^{stop}\|/\sqrt{h}$, respectively.
In Table ~\ref{tab:iteration numbers for weak stopping criterion}
the corresponding results are shown and we infer that the ratios do not remain bounded as $h \to 0$
indicating suboptimal approximations. Comparing the results with those reported in
Tables ~\ref{tab:iteration numbers for obstacle residual h2} and ~\ref{tab:iteration numbers for rof residual h}
we conclude that in order to obtain an accurate approximation
one has to control both the primal iterates ~$Bu_h^j$ and the dual iterates ~$\l_h^j$.

\begin{table}[!htb]
\centering
\begin{tabular}{|C{0.8cm}||C{0.8cm}|C{1.5cm}||C{0.8cm}|C{1.5cm}||C{0.8cm}|C{1.5cm}|}\firsthline
\multicolumn{1}{|c||}{} & \multicolumn{6}{c|}{Obstacle ($\veps_{stop}^{(2)}=h^2$, $\overline{\tau}=h^{-3}$)} \\[1pt] \hline
\multicolumn{1}{|c||}{\rule{0pt}{0.4cm}} & \multicolumn{2}{c||}{ADMM} & \multicolumn{2}{c||}{Fast-ADMM} & \multicolumn{2}{c|}{Variable-ADMM}\\[1pt] \hline
\multicolumn{1}{|c||}{\rule{0pt}{0.4cm}$\ell$} & $N$ & $E_h/h$ & $N$ & $E_h/h$ & $N$ & $E_h/h$ \\[1pt] \hline
\multicolumn{1}{|c||}{\rule{0pt}{0.4cm}$3$} &$   23$ & 0.04 & $   14$ & 0.0296& $   10$ & 0.147\\
$4$ &$  152$ & 0.164 & $   43$ & 0.0321& $   37$ & 0.123\\
$5$ &$ 1123$ & 0.453 & $  120$ & 0.142& $   47$ & 0.188\\
$6$ &$ 2785$ & 15.1 & $  164$ & 9.46& $   96$ & 0.173\\
$7$ &$ 8455$ & 50.9 & $-$ & 57.9 & $  168$ & 0.143\\ \hline
\multicolumn{1}{|c||}{\rule{0pt}{0.4cm}} & \multicolumn{6}{c|}{ROF ($\veps_{stop}^{(2)}=h$, $\overline{\tau}=1$)} \\[1pt] \hline
\multicolumn{1}{|c||}{\rule{0pt}{0.4cm}$\ell$}& $N$ & $E_h/\sqrt{h}$ & $N$ & $E_h/\sqrt{h}$ & $N$ & $E_h/\sqrt{h}$ \\[1pt] \hline
\multicolumn{1}{|c||}{\rule{0pt}{0.4cm}$3$} & $    6$ & 1.19 & $    5$ & 1.10 & $    6$ & 1.19\\
$4$ & $   11$ & 2.10 & $    7$ & 2.09 & $   11$ & 2.10\\
$5$ & $   24$ & 3.20 & $   11$ & 3.24 & $   24$ & 3.20\\
$6$ & $   54$ & 4.81 & $   19$ & 4.76 & $   54$ & 4.81\\
$7$ & $  123$ & 7.13 & $   31$ & 7.04 & $  123$ & 7.13\\ \lasthline
\end{tabular}
\caption{Iteration numbers and ratios $E_h/h$ and $E_h/\sqrt{h}$ with stopping criterion $\|\l_h^j-\l_h^{j-1}\| \leq h^2/\til{C}_0$ and $\tau_j\|\nabla(u_h^j-u_h^{j-1})\| \leq h/\til{C}_0$ and step size $\overline{\tau}=h^{-3}$ and $\overline{\tau}=1$ for Example ~\ref{ex:obstacle} and Example ~\ref{ex:rof}, respectively, using ADMM, Fast-ADMM and Variable-ADMM.}\label{tab:iteration numbers for weak stopping criterion}
\end{table}

\section{Conclusion}\label{sec:conclusion}

From our numerical experiments we infer the following observations:
\begin{itemize}
\item The Variable-ADMM can considerably improve the performance for any initial step size. If not improving the Variable-ADMM at least yields results comparable to those of ADMM and Fast-ADMM which differ by a fixed factor.
\item For large initial step sizes, i.e., $\overline{\tau}=h^{-2}$, the Variable-ADMM always yields lower iteration numbers than the other two schemes. This suggests to choose a large initial step size.
\item The reinitialization of the Variable-ADMM did not considerably influence the total number of iterations.
\item In order to obtain meaningful approximations one has to control both contributions to the residual, i.e., both
~$Bu^j$ and ~$\l^j$ have to be controlled which is accomplished by $R_j \leq \veps_{stop}/C_0$.
\end{itemize}

\appendix

\section{Co-coercivity of ~$\nabla G$}\label{apdx:co-coercivity}

For completeness we include a short proof of the co-coercivity estimate needed in Theorem ~\ref{thm:lin conv},
see, e.g., ~\cite{hl:04}.

\begin{lemma}\label{lemma:co-coercivity}
Assume that ~$G$ is convex and differentiable such that $\nabla G$ is Lipschitz continuous
with Lipschitz constant ~$L_G$. For all $v,v' \in X$ it holds
\[
L_G (v-v',\nabla G(v)-\nabla G(v'))_X \geq \|\nabla G(v)-\nabla G(v')\|_X^2.
\]
\end{lemma}

\begin{proof}
Let $v,v' \in X$. We define
\[
G_v(w) = G(w) - (\nabla G(v),w)_X, \quad \text{and} \quad g_v(w) = \frac{L_G}{2}\|w\|_X^2 - G_v(w).
\]
The functional $G_v$ is convex since it is the sum of two convex functionals.
Furthermore, one can check that for all $w,w'$ we have
\[
(\nabla g_v(w)-\nabla g_v(w'),w-w')_X \geq 0,
\]
i.e., ~$g_v$ is convex. Thus, we have for all $w,w' \in X$
\begin{align*}
g_v(w') &\geq (\nabla g(w),w'-w)_X + g(w) \\
\Longleftrightarrow \quad G_v(w') &\leq G_v(w) + (\nabla G_v(w),w'-w)_X + \frac{L_G}{2} \|w-w'\|_X^2. 
\end{align*}
Note that ~$v$ is a minimizer of ~$G_v$ since $\nabla G_v(v) = \nabla G(v) - \nabla G(v)=0$.
Hence, we have for all $w,w' \in X$
\begin{equation}\label{G_v}
G_v(v) \leq G_v(w) + (\nabla G_v(w),w'-w)_X + \frac{L_G}{2} \|w-w'\|_X^2.
\end{equation}
By minimizing the right-hand side of ~\eqref{G_v} with respect to ~$w'$ for fixed ~$w$ we obtain
the critical point $w^*=w-\frac{1}{L_G}\nabla G_v(w)$. Choosing $w'=w^*$ in ~\eqref{G_v} we obtain
\[
\frac{1}{2 L_G} \|\nabla G_v(w)\|_X^2 \leq G_v(w) - G_v(v).
\]
This yields
\begin{align*}
G(v')-G(v)-&(\nabla G(v),v'-v)_X = G_v(v')-G_v(v) \\
&\geq \frac{1}{2 L_G} \|\nabla G_v(v')\|_X^2 = \frac{1}{2 L_G} \|\nabla G(v')-\nabla G(v)\|_X^2.
\end{align*}
Exchanging roles of ~$v$ and ~$v'$ and adding the inequalities yields the assertion.
\end{proof}

\section{Fast-ADMM}\label{apdx:fast-admm}

In ~\cite{gosb:14} an accelerated version of ADMM with fixed step sizes is proposed.
The work is inspired by an acceleration technique presented
in ~\cite{nesterov:83} which has also been used in the context of forward-backward splitting
in ~\cite{bt:09} and in ~\cite{gms:13} for the special case $B=I$.
The technique consists in a certain extrapolation of the variables.
The authors can prove a $\mathcal{O}(1/J^2)$ convergence rate for the objective value of the dual problem
if ~$F$ and ~$G$ are strongly convex, ~$G$ is quadratic and if the step size is chosen properly. The residual
also enjoys this convergence rate if, in addition, ~$B$ has full row rank. However, for problems with ~$F$ or ~$G$
being only convex they have to impose a restart condition to guarantee stability and convergence of the method. We
will refer to this method with included restart condition as the Fast-ADMM. The details of the algorithm are
given in the following.

\begin{algorithm}[Fast-ADMM]\label{alg:fast_admm}
Choose $(u^0,\l^0)\in X\times Y$ such that $G(u^0)<\infty$.
Choose $\overline{\tau}>0$, $\g \in (0,1)$, $\overline{R} \gg 0$ and $\veps_{stop}>0$. Set
$\hu^0=u^0$, $\widehat{\l}^0=\l^0$, $\theta_0=1$ and $j=1$. \\
(1) Set $R_0=\overline{R}$.\\
(2)-(6) As in Algorithm ~\ref{alg:alg_2}.\\
(7) If $R_j < \g R_{j-1}$ set
\begin{align*}
\theta_j&=\frac{1+\sqrt{1+4\theta_{j-1}^2}}{2},\\
\hu^j &= u^j + \frac{\theta_{j-1}-1}{\theta_j}(u^j-u^{j-1}),\\
\widehat{\l}^j &= \l^j \frac{\theta_{j-1}-1}{\theta_j}(\l^j-\l^{j-1}).
\end{align*}
Otherwise, set $\theta_j=1$, $\hu^j=u^{j-1}$, $\widehat{\l}^j=\l^{j-1}$
and $R_j=\g^{-1}R_{j-1}$.\\
(8) Set $j\to j+1$ and continue with~(2). 
\end{algorithm}

\medskip
\textit{Acknowledgements.} The authors acknowledge financial support for the project "Finite Element approximation of functions of bounded variation and application to model of damage, fracture and plasticity" (BA 2268/2-1) by the DFG via the priority program "Reliable Simulation Techniques in Solid Mechanics, Development of Non-standard Discretization Methods, Mechanical and Mathematical Analysis" (SPP 1748).

\end{document}

%% file: includes/sb_macros.tex
\def\R{\mathbb{R}}

\def\N{\mathbb{N}}

\def\cL{\mathcal{L}}

\def\cN{\mathcal{N}}

\def\cS{\mathcal{S}}
\def\cT{\mathcal{T}}

\def\a{\alpha}
\def\b{\beta}
\def\g{\gamma}
\def\d{\delta}

\def\l{\lambda}

\def\p{\partial}

\def\veps{\varepsilon}
\def\vrho{\varrho}

\def\O{\Omega}

\def\hu{\widehat{u}}

\newcommand{\dv}[1]{\,{\mathrm d}#1}

\newcommand{\wcheck}[1]{#1\hspace{-.8ex}\mbox{\huge {\lower.45ex \hbox{$\textstyle \check{}$}}} \hspace{.5ex}}

\DeclareMathOperator{\diver}{div}

\let\oldmarginpar\marginpar
\renewcommand\marginpar[1]{
  \oldmarginpar[\raggedleft\footnotesize #1]
  {\raggedright\footnotesize #1}}
\usepackage{ifthen}
\usepackage[normalem]{ulem}

\newtheorem{definition}{Definition}
\newtheorem{lemma}[definition]{Lemma}
\newtheorem{proposition}[definition]{Proposition}
\newtheorem{theorem}[definition]{Theorem}
\newtheorem{corollary}[definition]{Corollary}
\newtheorem{remark}[definition]{Remark}
\newtheorem{remarks}[definition]{Remarks}
\newtheorem{example}[definition]{Example}

\newtheorem{algorithm}[definition]{Algorithm}
\numberwithin{definition}{section}

\AtAppendix{\counterwithin{definition}{section}}
\definecolor{tourquoise}{RGB}{0,170,180}	
\definecolor{darkred}{RGB}{238,34,34}		
\definecolor{darkgreen}{RGB}{0,190,0}		
\definecolor{lightgray}{RGB}{210,210,210}	
\definecolor{deepblue}{RGB}{0,0,240}		
\definecolor{darkgray}{RGB}{144,144,144}	
\definecolor{kingblue}{RGB}{64,96,224}		
\definecolor{gold}{RGB}{240,208,0}		
\definecolor{verydarkred}{RGB}{176,0,0}		

%% file: Step_size_control.bbl
\begin{thebibliography}{99}

\bibitem{sbartels:15}
     \newblock S. Bartels,
     \newblock \emph{Numerical Methods for Nonlinear Partial Differential Equations},
     \newblock Springer, Heidelberg, 2015.

\bibitem{bns:14} 
	\newblock S. Bartels, R. H. Nochetto and A. J. Salgado,
	\newblock Discrete total variation flows without regularization,
	\newblock \emph{SIAM J. Numer. Anal.}, \textbf{52} (2014), 363--385.

\bibitem{bt:09} 
	\newblock A. Beck and M. Teboulle,
	\newblock A fast iterative shrinkage-thresholding algorithm for linear inverse problems,
	\newblock \emph{SIAM Journal of Imaging Sciences}, \textbf{2} (2009), 183--202.

\bibitem{bv:04}
     \newblock S. Boyd and L. Vandenberghe,
     \newblock \emph{Convex Optimization},
     \newblock University Press, Cambridge, 2004.

\bibitem{bs:08}
     \newblock S. C. Brenner and L. R. Scott,
     \newblock \emph{The Mathematical Theory of Finite Element Methods},
     \newblock 3rd edition, Texts in Applied Mathematics, vol. 15, Springer, New York, 2008.

\bibitem{bs:68} 
	\newblock H. Br\'{e}zis and G. Stampacchia,
	\newblock Sur la r\'{e}gularit\'{e} de la solution d'in\'{e}quations elliptiques,
	\newblock \emph{Bulletin de la S. M. F.}, \textbf{96} (1968), 153--180.

\bibitem{chambolle:04} 
	\newblock A. Chambolle,
	\newblock An algorithm for total variation minimization and applications, 
	\newblock \emph{J. Math. Imaging Vis.}, \textbf{20} (2004), 89--97.

\bibitem{cp:11} 
	\newblock A. Chambolle and T. Pock,
	\newblock A first-order primal-dual algorithm for convex problems with applications to imaging, 
	\newblock \emph{J. Math. Imaging Vis.}, \textbf{40} (2011), 120--145.


\bibitem{dhyz:16} 
	\newblock Y.-H. Dai, D. Han, X. Yuan and W. Zhang,
	\newblock A sequential updating scheme of the Lagrange multiplier for separable convex programming, 
	\newblock \emph{Math. Comp.}, \textbf{86} (2016), 315--343.

\bibitem{dy:16}
	\newblock W. Deng and W. Yin,
	\newblock On the global and linear convergence of the generalized alternating direction method of multipliers, 
	\newblock \emph{J. Sci. Comput.}, \textbf{66} (2016), 889--916.

\bibitem{dr:56}
	\newblock J. Douglas and H. Rachford,
	\newblock On the numerical solution of heat conduction problems in two and three space variables, 
	\newblock \emph{Trans. Am. Math. Soc.}, \textbf{82} (1998), 421--439.

\bibitem{eckstein:89}
    \newblock J. Eckstein,
    \newblock \emph{Splitting methods for monotone operators with applications to parallel optimization},
    \newblock Ph.D thesis, Massachusetts Institute of Technology, 1989.

\bibitem{fg:83}
     \newblock M. Fortin and R. Glowinski,
     \newblock \emph{Augmented Lagrangian Methods},
     \newblock 1$^{st}$ edition, North-Holland Publishing Co., Amsterdam, 1983.

\bibitem{gabay:83} 
	\newblock D. Gabay,
	\newblock Applications of the method of multipliers to variational inequalities,
	\newblock in \emph{Augmented Lagrangian Methods: Applications to the Solution of Boundary-Value Problems} (eds. M. Fortin and R. Glowinski), (1983), 299--331.

\bibitem{gm:76}
	\newblock D. Gabay and B. Mercier,
	\newblock A dual algorithm for the solution of nonlinear variational problems via finite element approximation, 
	\newblock \emph{Comp. \& Maths. with Appls.}, \textbf{2} (1976), 17--40.

\bibitem{glowinski:84}
     \newblock R. Glowinski,
     \newblock \emph{Numerical Methods for Nonlinear Variational Problems},
     \newblock Springer, New York, 1984.

\bibitem{gl:89}
     \newblock R. Glowinski and P. Le Tallec,
     \newblock \emph{Augmented Lagrangians and Operator-Splitting Methods in Nonlinear Mechanics},
     \newblock SIAM, Philadelphia, 1989.

\bibitem{gm:75}
	\newblock R. Glowinski and A. Marroco,
	\newblock Sur l'approximation par \'{e}l\'{e}ments finis d'ordre un, et la r\'{e}solution, par p\'{e}nalisation-dualit\'{e} d'une classe de probl\`{e}mes de Dirichlet non lin\'{e}aires, 
	\newblock \emph{Revue fran\c{c}aise d'automatique, informatique, recherche op\'{e}rationelle. Analyse num\'{e}rique}, \textbf{9} (1975), 41--76.

\bibitem{gms:13}
	\newblock D. Goldfarb, S. Ma and K. Scheinberg,
	\newblock Fast alternating linearization methods for minimizing the sum of two convex functions, 
	\newblock \emph{Math. Program.}, \textbf{141} (2013), 349--382.

\bibitem{gosb:14}
	\newblock T. Goldstein, B. O'Donoghue, S. Setzer and R. Baraniuk,
	\newblock Fast alternating direction optimization methods, 
	\newblock \emph{SIAM J. Imaging Sci.}, \textbf{7} (2014), 1588--1623.

\bibitem{go:09}
	\newblock T. Goldstein and S. Osher,
	\newblock The split Bregman method for L1 regularized problems, 
	\newblock \emph{SIAM J. Imaging Sci.}, \textbf{2} (2009), 323--343.

\bibitem{hl:04}
     \newblock J.-B. Hiriart-Urruty and C. Lemar\'{e}chal,
     \newblock \emph{Fundamentals of Convex Analysis},
     \newblock Springer, Berlin, Heidelberg, 2004.

\bibitem{hyz:14}
	\newblock D. Han, X. Yuan and W. Zhang,
	\newblock An augmented Lagrangian based parallel splitting method for separable convex minimization with applications to image processing, 
	\newblock \emph{Math. Comp.}, \textbf{83} (2014), 2263--2291.

\bibitem{hlhy:02}
	\newblock B. He, L.-Z. Liao, D. Han and H. Yang,
	\newblock A new inexact alternating directions method for monotone variational inequalities, 
	\newblock \emph{Math. Program.}, \textbf{92} (2002), 103--118.

\bibitem{hlw:03} 
	\newblock B. He, L.-Z. Liao and S. Wang,
	\newblock Self-adaptive operator splitting methods for monotone variational inequalities, 
	\newblock \emph{Numer. Math.}, \textbf{94} (2003), 715--737.

\bibitem{hyw:00} 
	\newblock B. S. He, H. Yang and S. L. Wang,
	\newblock Alternating direction method with self-adaptive penalty parameters for monotone variational inequalities, 
	\newblock \emph{Journal of Optimization and Applications}, \textbf{106} (2000), 337--356.

\bibitem{hy:98}
	\newblock B. He and H. Yang,
	\newblock Some convergence properties of a method of multipliers for linearly constrained monotone variational inequalities, 
	\newblock \emph{Operations Research Letters}, \textbf{23} (1998), 151--161.

\bibitem{hy:12}
	\newblock B. He and X. Yuan,
	\newblock On the $\mathcal{O}(1/n)$ convergence rate of the Douglas-Rachford alternating direction method, 
	\newblock \emph{SIAM J. Numer. Anal.}, \textbf{50} (2012), 700--709.

\bibitem{hy:15}
	\newblock B. He and X. Yuan,
	\newblock On non-ergodic convergence rate of Douglas-Rachford alternating direction method of multipliers, 
	\newblock \emph{Numerische Mathematik}, \textbf{130} (2015), 567--577.

\bibitem{hestenes:69}
	\newblock M. R. Hestenes,
	\newblock Multiplier and gradient methods, 
	\newblock \emph{Journal of Optimization Theory and Applications}, \textbf{4} (1969), 303--320.

\bibitem{hik:03}
	\newblock M. Hinterm\"{u}ller, K. Ito and K. Kunisch,
	\newblock The primal-dual active set strategy as a semi-smooth Newton method, 
	\newblock \emph{SIAM J. Optim.}, \textbf{13} (2003), 865--888.

\bibitem{km:98}
	\newblock B. S. Kontogiorgis and R. R. Meyer,
	\newblock A variable-penalty alternating directions method for convex optimization, 
	\newblock \emph{Math. Program.}, \textbf{83} (1998), 29--53.

\bibitem{lm:79}
	\newblock P. L. Lions and B. Mercier,
	\newblock Splitting algorithms for the sum of two nonlinear operators, 
	\newblock \emph{SIAM J. Numer. Anal.}, \textbf{16} (1979), 964--979.

\bibitem{nesterov:83} 
	\newblock Y. Nesterov,
	\newblock A method for solving a convex programming problem with convergence rate $o(1/k^2)$,
	\newblock \emph{Soviet Math. Dokl.}, \textbf{27} (1983), 372--376.

\bibitem{nesterov:05} 
	\newblock Y. Nesterov,
	\newblock Smooth minimization of non-smooth functions,
	\newblock \emph{Math. Program.}, \textbf{103} (2005), 127--152.

\bibitem{powell:69} 
	\newblock M. J. D. Powell,
	\newblock A method for nonlinear constraints in minimization problems,
	\newblock in \emph{Optimization, Academic Press} (ed. R. Fletcher), (1969), 283--298.

\bibitem{rockafellar:70}
     \newblock R. T. Rockafellar,
     \newblock \emph{Convex Analysis},
     \newblock Princeton University Press, New Jersey, 1970.

\bibitem{rockafellar:76}
	\newblock R. T. Rockafellar,
	\newblock Monotone operators and the proximal point algorithm, 
	\newblock \emph{SIAM J. Control. Optim.}, \textbf{14} (1976), 877--898.

\bibitem{rof:92} 
	\newblock L. I. Rudin, S. Osher and E. Fatemi,
	\newblock Nonlinear total variation based noise removal algorithms, 
	\newblock \emph{Physica D}, \textbf{60} (1992), 259--268.

\bibitem{sx:14} 
	\newblock Y. Shen and M. Xu,
	\newblock On the $\mathcal{O}(1/t)$ convergence rate of Ye-Yuan's modified alternating direction method of multipliers, 
	\newblock \emph{Applied Mathematics and Computation}, \textbf{226} (2014), 367--373.


\bibitem{yy:07} 
	\newblock C. Ye and X. Yuan,
	\newblock A descent method for structured monotone variational inequalities, 
	\newblock \emph{Optim. Methods Softw.}, \textbf{22} (2007), 329--338.

\end{thebibliography}
